\documentclass[12pt]{amsart}

\usepackage[shortlabels]{enumitem}

\usepackage{amsfonts}
\usepackage{amsmath, amsrefs, amsthm, amssymb}
\usepackage[active]{srcltx}		
\usepackage{pdfsync}					
\usepackage{graphicx}
\usepackage{bbm}
\usepackage{amsfonts}
\usepackage{amsthm}
\usepackage{graphicx}
\usepackage{framed}
\usepackage{multicol,multienum}
\usepackage{graphicx} 
\usepackage{booktabs} 
\usepackage{mathrsfs}
\usepackage{tcolorbox}
\usepackage{bm}
\usepackage{subfigure}
\usepackage{epstopdf}

\newtheorem{theorem}{Theorem}[section]

\newtheorem{lemma}[theorem]{Lemma}

\theoremstyle{definition}
\newtheorem{definition}[theorem]{Definition}

\newtheorem{remark}[theorem]{Remark}

\numberwithin{equation}{section}

\begin{document}
\title []{New results on the global solvability and blow-up for a class of weakly dissipative Camassa-Holm equations}\thanks{This work was partially supported by the National Natural Science Foundation
of China (No. 12231008).}

\author{Lei Zhang, Bin Liu}
\address{School of Mathematics and Statistics, Hubei Key Laboratory of Engineering Modeling  and Scientific Computing, Huazhong University of Science and Technology  Wuhan 430074, Hubei, P.R. China.}
\email{lei\_zhang@hust.edu.cn}
\email{binliu@mail.hust.edu.cn}

\date{\today}

\begin{abstract}
In this paper, we consider the Cauchy problem for a class of weakly dissipative Camassa-Holm  equations in nonhomogeneous Besov spaces. First, we prove that the Cauchy problem admits a unique global strong solution in Besov spaces with proper condition on the dissipation parameter $\lambda>0$. The novel ingredients in the proof lies in transforming the equations into a class of damped Camassa-Holm equations, and performing a non-standard iterative method. It is shown that our result holds for the damped equations with more general time-dependent parameters, which improves the existed results from Sobolev spaces to Besov spaces without assuming any sign condition on the initial data. Second, we derive two kinds of blow-up criteria in suitable Sobolev spaces, which in some sense inform us how the dissipation parameter $\lambda$ influences the singularity formation of strong solutions.


\end{abstract}

\maketitle
\section{Introduction and main results}

In this paper, we consider the Cauchy problem of the following weakly dissipative two-component $b$-family equations over $\mathbb{K}=\mathbb{R}$ or $\mathbb{K}=\mathbb{T}=\mathbb{R}/2\pi\mathbb{Z}$ (cf. \cite{lenells2013weakly})
\begin{equation}\label{1.1}
\begin{cases}
 m_t+um_x+bu_xm +\kappa \sigma\sigma_x+\lambda m=0,& t>0,~x\in \mathbb{K},\\
 m=u-u_{xx},& t>0,~x\in \mathbb{K},\\
 \sigma_t+(u\sigma)_x+\lambda\sigma=0,& t>0,~x\in \mathbb{K},\\
 m(0,x)=m_0(x),~~\sigma(0,x)=\sigma(x),&  x\in \mathbb{K},
\end{cases}
\end{equation}
where $\lambda>0$ is the dissipative parameter, $\kappa=\pm 1$. The real dimensionless constant $b\in \mathbb{R}$ is a parameter that provides the competition, or balance, in fluid convection between nonlinear steepening and amplification due to stretching. The number $b$ is also the number of covariant dimensions associated with the momentum density $m$. The unknown $u(t, x)$ in \eqref{1.1} stands for the horizontal velocity of the fluid, and $\sigma(t, x)$  is related to the free surface elevation from equilibrium with the boundary assumption.

The local and global well-posedness, blow-up criteria for the system \eqref{1.1} without the dissipative terms, i.e., $\lambda=0$, have been studied by several authors, see for example \cites{liu2011cauchy,zong2012properties,zhu2013cauchy,gui2010global,guan2010global}. Note that the system \eqref{1.1} can be reduced to some well-known weakly dissipative Camassa-Holm type equations. For example, the system \eqref{1.1} reduces to the weakly dissipative Camassa-Holm equation when $b = 2$ and $\sigma(t, x)\equiv 0$ \cites{guo2008blow,wu2009global,novruzov2012blow,shen2011time}. The local and global well-posedness in $H^s$ with $s >  \frac{3}{2}$ are established in \cite{wu2009global}. When $b = 3$ and $\sigma(t, x)\equiv0$, system \eqref{1.1} becomes the weakly dissipative Degasperis-Procesi equation \cites{guo2011global,guo2009some,wuescher2009global,wuyin2008blow}. Moreover, if the second
component $\sigma(t, x)$ in \eqref{1.1} is assumed to be nonzero, the choices $b = 2$ and $b = 3$ in \eqref{1.1} give rise to the weakly dissipative two-component Camassa-Holm and Degasperis-Procesi equations, respectively \cites{hu2011global,wuyin2008blow}.

The dissipative system  \eqref{1.1}  was first studied by Lenells and Wunsch \cite{lenells2013weakly}, in which the authors shown that the pair $(u,\sigma)$ is a solution to the weakly dissipative system \eqref{1.1} if and only if the pair $(v,\rho)$ defined by
\begin{equation}\label{1.2}
\begin{split}
u(t,x)=e^{-\lambda t} v\left(\frac{1-e^{-\lambda t}}{\lambda},x\right),\quad \sigma(t,x)=e^{-\lambda t} \rho\left(\frac{1-e^{-\lambda t}}{\lambda},x\right)
\end{split}
\end{equation}
is a solution to the associated non-dissipative system. It is shown that the similar result also holds for the Novikov equation with cubic nonlinearity \cite{lenells2013weakly}. In terms of the exponentially time-dependent scaling \eqref{1.2}, the authors in \eqref{1.1} obtained the global existence result for the damped system \eqref{1.1} in Sobolev spaces. Since the existence of solutions to \eqref{1.1} is based on the transformation \eqref{1.2} and the known results for the corresponding non-dissipative equations, the proof depends on the sign condition on initial momentum, e.g. $m(0,x)\geq 0$,  the proper blow-up criteria at hand as well as the method of characteristics.

The \emph{main contribution} of this paper is to prove the global Hadamard well-posedness of strong solutions to \eqref{1.1} in  larger Besov spaces, and analyse the influence of dissipative parameters $\lambda >0$ in the singular formulation of solutions. To this end, instead of the transformation \eqref{1.2}, we  take the following change of variables:
\begin{equation}\label{1.3}
\begin{split}
\widetilde{m}(t,x)=e^{\lambda t}m(t,x),\quad \widetilde{\sigma}(t,x)=e^{\lambda t}\sigma(t,x),
\end{split}
\end{equation}
then the Cauchy problem \eqref{1.1} can be reformulated as
\begin{equation}\label{1.4}
\begin{cases}
 \widetilde{m}_t+e^{-2\lambda t}\widetilde{u}\widetilde{m}_x+be^{-2\lambda t}\widetilde{u}_x\widetilde{m} +\kappa e^{-2\lambda t}\widetilde{\sigma}\widetilde{\sigma}_x=0,& t>0,~x\in \mathbb{K},\\
  \widetilde{\sigma}_t+e^{-2\lambda t}(\widetilde{u}\widetilde{\sigma})_x=0,& t>0,~x\in \mathbb{K},\\
 \widetilde{m}(0,x)=m_0(x),\quad\widetilde{\sigma}(0,x)=\sigma(x),&  x\in \mathbb{K}.
\end{cases}
\end{equation}

The transformed system \eqref{1.4} can be regarded as a damped two-component $b$-family equations with variable coefficients $e^{-2\lambda t}$. At this stage, it is interesting to mention that such a type of equation, called the  nonisospectral Camassa-Holm equations, has recently been derived by Chang et al. \cites{Chang2014,Chang2018}, which has the common characteristic that the parameters in the equations are no longer real numbers but time-dependent functions. Especially, in \cite{zhang2019}, we established a global well-posedness result for the system derived in \cite{Chang2018} with suitable assumptions on the time-dependent parameters. The method used in the present paper is inspired by our previous work \cite{zhang2019}, where the key idea involved in our approach is motivated by the classical ODE theory. For example, one considers the initial value problem
$\frac{df(t)}{dt}+\lambda f(t) =g(t)$, $f(0)=f_0 $. If we take the change of variable $\widetilde{f}(t)=e^{\lambda t}f(t)$, then the equation reduces to $\frac{d\widetilde{f}(t)}{dt}=e^{\lambda t}g(t)$, and the solution $\widetilde{f}(t)$ can be obtained just by integrating over $[0,t]$, for all $t>0$. However, when we apply the transformation to the system  \eqref{1.4}, the quadratic nonlinear terms and the coupling nature of the system make the solvability  of solutions to be more difficult, and hence some new ideas have to be explored.

To make our result more general, in the following we shall consider the following system:
\begin{equation}\label{1.5}
\begin{cases}
 m_t+\alpha(t)um_x+\beta(t)u_xm +\gamma(t)\sigma\sigma_x=0,& t>0,~x\in \mathbb{K},\\
 \sigma_t+\xi(t)(u\sigma)_x=0,& t>0,~x\in \mathbb{K},\\
 m(0,x)=m_0(x),\quad
 \sigma(0,x)=\sigma(x),&  x\in \mathbb{K},
\end{cases}
\end{equation}
where $\alpha,\beta,\gamma$ and $\xi$ are time-dependent parameters satisfying integrable conditions.

Set $p(x)=\frac{1}{2}e^{-|x|}$, we have
$$
(1-\partial_x^2)^{-1}f=p*f,\quad\textrm{for all}~~ f\in L^2(\mathbb{R}).
$$
Then the system  \eqref{1.5} can be rewritten as
\begin{equation}\label{1.6}
\begin{cases}
 u_t+\alpha(t) uu_x=-\partial_xp*\left(\frac{\beta(t)}{2} u^2+\frac{\gamma(t)}{2} \sigma^2+\frac{\beta(t)-3\alpha(t)}{2}u_x^2\right),& t>0,~x\in \mathbb{K},\\
 \sigma_t+\xi(t)u\sigma_x=-\xi(t)\sigma u_x,& t>0,~x\in \mathbb{K},\\
 u(0,x)=u_0(x),\quad
 \sigma(0,x)=\sigma_0(x),&  x\in \mathbb{K}.
\end{cases}
\end{equation}

The main result of the global well-posedness for the Cauchy problem  \eqref{1.6} (or \eqref{1.5}) in Besov spaces is enunciated by the following theorem.
\begin{theorem} \label{theorem1}
Let $s>\max\{1+\frac{1}{p},\frac{3}{2}\}$ and $p,r\in [0,\infty]$. Assume that $(u_0,\rho_0)\in B_{p,r}^s(\mathbb{K})\times B_{p,r}^{s-1}(\mathbb{K})$, and the time-dependent parameters $\alpha,\beta,\gamma,\xi\in L^1([0,\infty);\mathbb{R})$ such that
\begin{equation*}
\begin{split}
\int_0^\infty(|\alpha(t')|+|\beta(t')|+|\gamma(t')|+|\xi(t')|) dt'\leq \frac{\ln2}{6C^2 h\left(\|u_0\|_{B^{s}_{p,r}}+\|\sigma_0\|_{B^{s-1}_{p,r}}\right)},
\end{split}
\end{equation*}
where
\begin{equation*}
\begin{split}
h(x)\overset{\text{def}}{=} & e^{ 2C^2x\int_0^{\infty}(|\alpha(t')|+|\xi(t')|)dt'}\\
&\times\left(x+ 4C^2x^2 \int_0^{\infty}(|\alpha(t')|+|\beta(t')|+|\gamma(t')|+|\xi(t')|)dt'\right).
\end{split}
\end{equation*}
Then for any $ T> 0$, the system \eqref{1.6} has a unique global strong solution
$$
(u,\sigma)\in
E_{p,r}^s(T)\overset{\text{def}}{=} C([0,T];B_{p,r}^{s}(\mathbb{K})\times B_{p,r}^{s-1}(\mathbb{K}))\bigcap C^1([0,T];B_{p,r}^{s-1}(\mathbb{K})\times B_{p,r}^{s-2}(\mathbb{K})).
$$
Moreover, the data-to-solution map $(u_0,\sigma_0)\mapsto (u,\sigma)$ is continuous from a neighborhood of $(u_0,\sigma_0)$  in $B_{p,r}^s(\mathbb{K})\times B_{p,r}^{s-1}(\mathbb{K})$ to
 $E_{p,r}^{s'}(T),$
for $s'<s$ when $r = \infty$ and $s'=s$ whereas $r < \infty$.
\end{theorem}

\begin{remark}
\begin{itemize}[leftmargin=0.62cm]
\item [1)]
Observing that the function
$$
f(x)\overset{\text{def}}{=}  6C^2 x e^{ 2C^2(\|u_0\|_{B^{s}_{p,r}}+\|\sigma_0\|_{B^{s-1}_{p,r}})x}\left(\|u_0\|_{B^{s}_{p,r}}+\|\sigma_0\|_{B^{s-1}_{p,r}}+ 4C^2(\|u_0\|_{B^{s}_{p,r}}+\|\sigma_0\|_{B^{s-1}_{p,r}})^2 x\right)
$$
is an increasing and continuous function for all  $x\in [0,\infty)$ with $f(0)=0$. The integrability condition provided in Theorem \ref{theorem1}  holds true as long as the $L^1$-integral $\|\alpha\|_{L^\infty}+\|\beta\|_{L^1}+\|\gamma\|_{L^1}+\|\xi\|_{L^1}$ is sufficiently small.

\item [2)]
Theorem \ref{theorem1} obtain a global existence in Besov spaces without using blow-up criteria and shape condition of the initial data, which improves the results in \cites{20,constantin1998global,constantin1998well, liu2006global, liuyin2006global}.
   \end{itemize}
  \end{remark}

Let us now come back to the Cauchy problem for the system \eqref{1.4}. If we take
$$
\alpha(t)=e^{-2\lambda t},~~\beta(t)=be^{-2\lambda t},~~\gamma(t)=\kappa e^{-2\lambda t},~~\xi(t)=e^{-2\lambda t},
$$
then we immediately conclude from Theorem \ref{theorem1} the following global existence theorem for the damped two-component $b$-family system \eqref{1.1}.
\begin{theorem} \label{theorem2}
Let $s>\max\{1+\frac{1}{p},\frac{3}{2}\}$ and $p,r\in [0,\infty]$. Assume that $(u_0,\rho_0)\in B_{p,r}^s\times B_{p,r}^{s-1}$, and dissipative constant $\lambda>0$ is sufficiently large such that
\begin{equation*}
\begin{split}
&e^{ \frac{2C^2}{\lambda}(\|u_0\|_{B^{s}_{p,r}}+\|\sigma_0\|_{B^{s-1}_{p,r}})}\left(\|u_0\|_{B^{s}_{p,r}}+\|\sigma_0\|_{B^{s-1}_{p,r}}+ \frac{2C^2 (2+|b|+|\kappa|)}{\lambda}(\|u_0\|_{B^{s}_{p,r}}+\|\sigma_0\|_{B^{s-1}_{p,r}})^2\right)\\
&\quad  \leq \frac{\lambda\ln2}{3C^2 (2+|b|+|\kappa|)}.
\end{split}
\end{equation*}
Then the Cauchy problem \eqref{1.4} has a unique global strong solution $(u,\sigma)\in E_{p,r}^{s}(\infty)$.
\end{theorem}

\begin{remark}
As an example, we provide a sufficient condition which satisfies the inequality in Theorem \ref{theorem2}:
\begin{equation*}
\begin{split}
\lambda\geq\frac{8C^2 (2+|b|+|\kappa|)}{\ln 2}\left(\|u_0\|_{B^{s}_{p,r}}+\|\sigma_0\|_{B^{s-1}_{p,r}}\right).
\end{split}
\end{equation*}
Observing that, for any given $\lambda>0$, Theorem \ref{theorem2} with above inequality actually gives the small data global-in-time result for the system \eqref{1.4} in Besov spaces, where the initial data satisfies $\|u_0\|_{B^{s}_{p,r}}+\|\sigma_0\|_{B^{s-1}_{p,r}}\leq\frac{\lambda\ln 2}{8C^2 (2+|b|+|\kappa|)}$. More importantly, Theorem \ref{theorem2} casted off  the shape conditions of initial data, cf. \cites{wuescher2009global,wu2009global,wuyin2008blow}.
\end{remark}

Our second goal in the present paper is to investigate the finite time blow-up regime for the system \eqref{1.5} (when $\beta\equiv3\alpha$) in Sobolev spaces, which in some sense tell us how the time-dependent parameters affect the singularity formation. The first blow-up criteria is   given in the following theorem.
\begin{theorem} \label{theorem3}
Assume that $s> \frac{3}{2} $, $\beta(t)=3\alpha(t)$  in \eqref{1.6}, and the parameters $\alpha,\gamma,\xi \in L^2_{loc}([0,\infty);\mathbb{R})$. Let $T^*>0$ be the maximum existence time of solution $(u,\sigma)$ with respect to the initial data $(u_0,\sigma_0)\in H^s \times H^{s-1}$. If $T^*<\infty$, then
$$
T^*<\infty\Rightarrow \int_0^{T^*}  (|\alpha(t')|+|\gamma(t')|+|\xi(t')|)  \|\partial_x u(t')\|_{L^\infty}dt'=\infty,
$$
and the blow-up time $T^*$ is estimated as follows
\begin{equation*}
\begin{split}
T^*\geq  \sup \left\{t>0;~~\|\alpha\|_{L^1(0,t)}+\|\gamma \|_{L^1(0,t)} +\|\xi\|_{L^1(0,t)}<  \frac{1}{C(\|u_0\|_{H^{s}}+ \| \sigma_0\|_{H^{s-1}})}\right\}.
\end{split}
\end{equation*}
In addition, if the parameters $\alpha,\gamma,\xi  \in L^1 ([0,\infty);\mathbb{R})$ satisfies
$$
\|\alpha\|_{L^1(0,\infty)}+\|\gamma \|_{L^1(0,\infty)} +\|\xi\|_{L^1(0,\infty)}= \frac{1}{2C(\|u_0\|_{H^{s}}+ \| \sigma_0\|_{H^{s-1}})} ,
$$
then $(u,\sigma)$ is a global strong solution.
\end{theorem}

Theorem \ref{theorem3} indicates that if the time-dependent parameters $\alpha,\gamma$ and $\xi$ are locally integrable functions, then the strong solutions to the system \eqref{1.6} can still blow-up in finite time. Moreover, it is seen from Theorem \ref{theorem3} that the blow-up of solutions only depends on the parameters $\alpha,\gamma,\xi$ as well as the slope of $u$, but not on the slope of $\sigma$. Similar blow-up criteria has been established for the non-isospectral peakon system \cites{zhang2020periodic,zhangqiaocmp}.

The next theorem provides a precise blow-up criteria for the strong solution with suitable conditions on parameters.
\begin{theorem} \label{theorem4}
Let $s> \frac{3}{2}$ and  $\beta(t)=3\alpha(t)$ in \eqref{1.6}. Assume that the parameters $\alpha,\gamma,\xi$ belong to $ L^2_{loc}([0,\infty);\mathbb{R})$ and satisfy the following sign conditions:
$$
 \xi(t)\geq 0,\quad \alpha(t)+\gamma(t)+\xi(t)\geq 0 ,\quad \forall t\in [0,\infty).
$$
If $(u,\sigma)$ solves the system \eqref{1.6} with initial data $(u_0,\sigma_0)\in H^s \times H^{s-1}$, then
$$
T^*<\infty \Leftrightarrow  \lim_{t\rightarrow T^*}\inf_{x\in \mathbb{K}} u_x(t,x)=-\infty.
$$
\end{theorem}

Theorem \ref{theorem4} demonstrate that the blow-up in the first component $u$ must occur before that in the second component $\sigma$ in finite time, which is similar to the well-known Camassa-Holm equation. However, the present result strongly depends on the integrability and sign conditions for the time-dependent parameters $\alpha,\gamma$ and $\xi$. Note that Theorem \ref{theorem4} can be applied to the damped two-component $b-$family system \eqref{1.1} (or \eqref{1.4}).

The remaining of this paper is organized as follows. In Section 2, we introduce the Littlewood-Paley theory, and some well-known results of the transport theory in Besov spaces. The Section  3 is devoted to the proof of the  global Hadamard well-posedness for the damped system \eqref{1.6} in Besov spaces. In Section 4, we first prove a crucial lemma regarding the $L^2$- and $L^\infty$-bound for solution $(u,\sigma)$, and then give the proofs of the blow-up criteria in Theorem \ref{theorem3} and Theorem \ref{theorem4}.

\vskip3mm\noindent
\textbf{Notation.} Throughout the paper, since all spaces of functions 
  are over  $\mathbb{K}$, we will drop them in the notation of function spaces if there is no any specific to be clarified.

\section{Preliminaries}\label{sec:prelim}
In this section, we recall some well-known facts of the Littlewood-Paley decomposition theory and the linear transport theory in Besov spaces. We introduce the Besov spaces defined on $\mathbb{R}$, the periodic case   can be defined similarly.

\begin{lemma} [\cite{64}]\label{lem:Dyaunit}
 Denote by $\mathcal {C}$ the annulus of centre 0, short radius $3/4$ and long radius $8/3$. Then there exits two positive radial functions $\chi$ and $\varphi$ belonging respectively to $C_c^\infty(B(0,4/3))$ and $ C_c^\infty(\mathcal {C})$ such that
 $$
 \chi(\xi)+\sum_{q\geq0}\varphi(2^{-q}\xi)=1, \quad \forall \xi \in \mathbb{R},
 $$
\end{lemma}

For the functions $\chi$ an $\varphi$ satisfying Lemma \ref{lem:Dyaunit}, we write $h=\mathcal {F}^{-1}\varphi$ and $\widetilde{h}=\mathcal {F}^{-1} \chi$, where $\mathcal {F}^{-1}u$ denotes the inverse Fourier transformation of $u$. The nonhomogeneous dyadic blocks $\Delta_j$ are defined by
$$
\Delta_qu\overset{\text{def}}{=}0,\quad q\leq-2;\quad \Delta_{-1}u\overset{\text{def}}{=} \chi(D)u=\int_{\mathbb{R}} u(x-y)\widetilde{h}(y)dy,\quad  q=-1;
$$
$$
\Delta_{q}u\overset{\text{def}}{=} \varphi(2^{-j}D )u=\int_{\mathbb{R}^d} u(x-y)h (2^{-j}y)dy,\quad  q\geq 0.
$$
The nonhomogeneous Littlewood-Paley  decomposition of $u\in \mathscr{S}' $ is given by
$$
u=\sum_{q\geq-1} \Delta_{q} u.
$$
We also define the high-frequence cut-off operator as
$$
S_qu\overset{\text{def}}{=}\sum\limits_{p\leq q-1 }\Delta _p u, \quad \forall q\in\mathbb{N}.
$$

\begin{definition} [\cite{64}]\label{def:PBesov}
For any $s\in \mathbb{R}$ and $p,r\in [1,\infty] $, the $d$-dimension nonhomogeneous Besov space $B_{p,r}^s  $ is defined by
$$
B_{p,r}^s \overset{\text{def}}{=}\left\{u\in \mathscr{S}' ; ~\|u\|_{B_{p,r}^s}=\|(2^{qs}\| \Delta _qu\|_{L^p})_{l\geq -1}\|_{l^r}  <\infty\right\},
$$
If $s=\infty$,  $B_{p,r}^\infty \overset{\text{def}}{=}\bigcap_{s\in\mathbb{R}}B_{p,r}^s $.
\end{definition}

Now let us state some useful results in the transport equation theory in Besov spaces,  which are crucial to the proofs of our main theorems.
\begin{lemma}[\cite{64}]\label{lem:priorieti}
Assume that $p,r\in [1,\infty]$ and  $s > -\frac{d}{p}$. Let $v$ be a vector field such that
 $\nabla v$ belongs to $L^1_T(B_{p,r}^{s-1})$ if $s>1+\frac{d}{p}$ or to $ L^1_T(B_{p,r}^{\frac{d}{p}}\cap L^\infty)$ otherwise. Suppose also that $f_0\in B_{p,r}^s$, $F\in L^1_T(B_{p,r}^s)$ and that  $f \in L^\infty_T(B_{p,r}^s)\cap C_T(\mathcal {S}')$ solves the linear transport equation
 \begin{eqnarray*}\label{eq:transeq}
(T)\quad\begin{cases}
\partial_tf+v\cdot\nabla f=F,\\
f|_{t=0}=f_0.
\end{cases}
\end{eqnarray*}
Then there is a constant $C$ depending only on $s,p$ and $r$ such that the following statements hold:

1) If $r=1$ or $s\neq 1+\frac{d}{p}$,  then
\begin{equation*}
\begin{split}
\|f(t)\|_{B_{p,r}^s}\leq  \|f_0\|_{B_{p,r}^s}+\int_0^t \|F(\tau)\|_{B_{p,r}^s}d\tau+C\int_0^t V'(\tau)\|f(\tau)\|_{B_{p,r}^s}d \tau,
\end{split}
\end{equation*}
or
\begin{equation}\label{2.1}
\begin{split}
\|f(t)\|_{B_{p,r}^s}\leq  e^{CV(t)}\left(\|f_0\|_{B_{p,r}^s}+\int_0^t e^{-CV(\tau)}\|F(\tau)\|_{B_{p,r}^s}d\tau\right)
\end{split}
\end{equation}
hold, where $ V(t)= \int_0^t\|\nabla v (\tau)\|_{B_{p,r}^{\frac{d}{p}}\bigcap L^\infty}d\tau$ if $s<1+\frac{d}{p}$ and $
V(t)=\int_0^t\|\nabla v(\tau)\|_{B_{p,r}^{s-1}}d\tau$ otherwise.

2) If $s\leq\frac{d}{p}$ and, in addition, $\nabla f_0\in L^\infty $, $\nabla f\in L^\infty _T(L^\infty)$ and $\nabla F\in L^1 _T(L^\infty)$, then
\begin{equation*}
\begin{split}
&\|f(t)\|_{B_{2,r}^s}+\|\nabla f(t)\|_{L^\infty}\\
&\quad \leq  e^{CV(t)}\left(\|f_0\|_{B_{2,r}^s}+\|\nabla f_0\|_{L^\infty}+\int_0^t e^{-CV(\tau)}(\|F(\tau)\|_{B_{2,r}^s}+\|\nabla F(\tau)\|_{L^\infty})d\tau\right).
\end{split}
\end{equation*}

3) If $f=v$, then for all $s>0$, the estimate \eqref{2.1} holds with $V(t)=\int_0^t\|\nabla v(\tau)\|_{L^\infty}d\tau$.

4) If $r<\infty$, then $f\in C_T(B_{2,r}^{s})$. If $r=\infty$, then $f\in C_T(B_{2,1}^{s'})$ for all $s'<s$.
\end{lemma}

\begin{lemma} [\cite{64}] \label{lem:transwell}
Let $(p,p_1,r)\in [1,\infty]^3$. Assume that $s>-d\min\{\frac{1}{p_1},\frac{1}{p'}\}$ with $p'=(1-\frac{1}{p})^{-1}$. Let $f_0\in B_{p,r}^s$ and $F\in L^1_T(B_{p,r}^s)$. Let $v \in L^\rho_T( B^{-M}_{\infty,\infty})$ for some $\rho > 1$, $M > 0$ and $\nabla v \in L^1_T(B^{\frac{d}{p_1}}_{p_1,\infty}\bigcap L^\infty)$ if $s < 1+\frac{d}{p_1}$, and $\nabla v  \in L^1_T( B^{s-1}_{p_1,r})$ if $s >  1+\frac{d}{p_1}$ or $s = 1+\frac{d}{p_1}$ and $r =1$. Then $(T)$ has a unique solution $f\in L^\infty_T( B^{s}_{p,r})\bigcap (\bigcap_{s'<s} C_T(B^{s'}_{p,1})$ and the inequalities in Lemma \ref{lem:priorieti} hold true. If, moreover, $r<\infty$, then we have $f\in  C_T(B^{s}_{p,r})$.
\end{lemma}

The following a priori estimates for the transport equation $(T)$ is crucial in the study of wave-breaking criteria.
\begin{lemma} (\cite{gui2010global})\label{lem:prioriSob}
 Let $0\leq s <1$. Suppose that $f\in L^\infty_T(H^s)\cap C_T(\mathcal {S}')$ is a solution to $(T)$ with velocity $v,\partial_xv\in L^1_T(L^\infty)$, initial data $f_0\in H^s$, $F\in L^1_T(H^s)$. Then, there exists a constant $C>0$ such that
$$
\|f(t)\|_{H^s}\leq  \|f_0\|_{H^s}+ C\int_0^t \|F(\tau)\|_{H^s}d\tau+C\int_0^t V'(\tau)\|f(\tau)\|_{H^s}d\tau,
$$
where
$$
V(t)= \int_0^t(\|v(\tau)\|_{L^\infty}+\|\partial_xv(\tau)\|_{L^\infty} )d\tau.
$$
\end{lemma}

\section{Global well-posedness in Besov spaces}

In this section, we mainly focus on the proof of Theorem \ref{theorem1}, which will be divided into the following several steps. Since the proof of Theorem \ref{theorem2} is completely same to Theorem \ref{theorem1}, we shall omit the details here.

\vskip3mm\noindent
\textbf{Step 0 (Approximate solutions):} We construct the approximate solutions via the Friedrichs iterative method. Starting from $(u^{(1)},\sigma^{(1)})\overset{\text{def}}{=} (S_1 u_0,S_1\sigma_0)$, we recursively define a sequence of functions
$(u^{(n)},\sigma^{(n)})_{n\geq1}$ by solving the following linear transport equations
\begin{equation}\label{3.1}
\begin{cases}
 u_t^{(n+1)}+\alpha(t) u^{(n)}u_x^{(n+1)}=-\partial_xp*\left(\frac{\beta(t)}{2} (u^{(n)})^2+\frac{\gamma(t)}{2} (\sigma^{(n)})^2+\frac{\beta(t)-3\alpha(t)}{2}(u_x^{(n)})^2\right),\\
 \sigma_t^{(n+1)}+\xi(t)u^{(n)}\sigma_x^{(n+1)}=-\xi(t)\sigma^{(n)} u_x^{(n)},\\
 u^{(n+1)}|_{t=0}=S_{n+1}u_0,\quad\sigma^{(n+1)}|_{t=0}=S_{n+1}\sigma_0.
\end{cases}
\end{equation}
Assuming that $(u^{(n)},\sigma^{(n)})\in E_{p,r}^s(T)$ for any positive $T $, it is seen that the right hand side of the system \eqref{3.1} belongs to $L^1_{loc}([0,\infty);B_{p,r}^s\times B_{p,r}^{s-1} )$ since the parameters $\alpha,\beta,\gamma,\xi \in L^1([0,\infty);\mathbb{R})$. Hence, applying  Lemma \ref{lem:transwell} ensures that the Cauchy problem \eqref{3.1} has a global solution $(u^{(n+1)},\sigma^{(n+1)})$ which bellongs to $E_{p,r}^s(T)$ for any positive $T $.

\vskip3mm\noindent
\textbf{Step 1 (Uniform bound):}
Applying the Lemma \ref{lem:priorieti} to \eqref{3.1}, we get
\begin{equation}\label{3.2}
\begin{split}
&\|u^{(n+1)}(t)\|_{B^{s}_{p,r}} \leq \|S_{n+1}u_0\|_{B^{s}_{p,r}}+C\int_0^t \|\alpha(t')\partial_x u^{(n)}(t')\|_{B^{s-1}_{p,r}}\|u^{(n+1)}(t')\|_{B^{s}_{p,r}} dt'\\
&\quad + \int_0^t \left\|\partial_xp*\left(\frac{\beta(t')}{2} (u^{(n)}(t'))^2+\frac{\gamma(t')}{2} (\sigma^{(n)}(t'))^2+\frac{\beta(t')-3\alpha(t')}{2}(u_x^{(n)}(t'))^2\right)\right\|_{B^{s}_{p,r}} dt',
\end{split}
\end{equation}
and
\begin{equation}\label{3.3}
\begin{split}
\|\sigma^{(n+1)}(t)\|_{B^{s-1}_{p,r}} \leq& \|S_{n+1}\sigma_0\|_{B^{s}_{p,r}}+C\int_0^t \|\xi(t')\partial_x u^{(n)}(t')\|_{B^{s-1}_{p,r}}\|\sigma^{(n+1)}(t')\|_{B^{s}_{p,r}} dt'\\
&+ \int_0^t \|\xi(t')\sigma^{(n)}(t') u_x^{(n)}(t')\|_{B^{s}_{p,r}} dt'.
\end{split}
\end{equation}
Using the fact that $(1-\partial_x^2)^{-1}$ is a $S^{-2}$-multiplier and the Besov space $B^{s}_{p,r}$ is a Banach algebra for any $s>\frac{1}{p}$, we deduce from \eqref{3.2} and \eqref{3.3} that
\begin{equation*}
\begin{split}
&\|u^{(n+1)}(t)\|_{B^{s}_{p,r}}+\|\sigma^{(n+1)}(t)\|_{B^{s-1}_{p,r}} \leq C\left(\|u_0\|_{B^{s}_{p,r}}+\|\sigma_0\|_{B^{s-1}_{p,r}}\right)\\
&\quad +C\int_0^t(|\alpha(t')|+|\xi(t')|)\|u^{(n)}(t')\|_{B^{s}_{p,r}}\left(\|u^{(n+1)}(t')\|_{B^{s}_{p,r}}+\|\sigma^{(n+1)}(t')\|_{B^{s}_{p,r}}\right) dt'\\
&  \quad+ C\int_0^t (|\alpha(t')|+|\beta(t')|+|\gamma(t')|+|\xi(t')|)\left(\|u^{(n)}(t')\|_{B^{s}_{p,r}}^2+ \|\sigma^{(n)}(t')\|_{B^{s-1}_{p,r}}^2\right) dt'.
\end{split}
\end{equation*}
An application of the Gronwall lemma  to above inequality leads to
\begin{equation}\label{3.4}
\begin{split}
&\|u^{(n+1)}(t)\|_{B^{s}_{p,r}}+\|\sigma^{(n+1)}(t)\|_{B^{s-1}_{p,r}}\\
&\quad \leq Ce^{C\int_0^t(|\alpha(t')|+|\xi(t')|)\|u^{(n)}(t')\|_{B^{s}_{p,r}}dt'}\bigg[\|u_0\|_{B^{s}_{p,r}}+\|\sigma_0\|_{B^{s-1}_{p,r}}\\
&\quad + C\int_0^t(|\alpha(t')|+|\beta(t')|+|\gamma(t')|+|\xi(t')|)\left(\|u^{(n)}(t')\|_{B^{s}_{p,r}}^2+ \|\sigma^{(n)}(t')\|_{B^{s-1}_{p,r}}^2\right)\bigg] dt'.
\end{split}
\end{equation}
Define
$$
H_0\overset{\text{def}}{=} H^{(n)}(0)=\|u_0\|_{B^{s}_{p,r}}+\|\sigma_0\|_{B^{s-1}_{p,r}},
$$
and
$$
H^{(n)}(t)\overset{\text{def}}{=} \|u^{(n)}(t)\|_{B^{s}_{p,r}}+\|\sigma^{(n)}(t)\|_{B^{s-1}_{p,r}},\quad \mbox{for all}~ n \geq 1.
$$
It follows  from the inequality \eqref{3.4} that
\begin{equation}\label{3.5}
\begin{split}
H^{(n+1)}(t) \leq &Ce^{C\int_0^t(|\alpha(t')|+|\xi(t')|)H^{(n)}(t')dt'}\\
&  \times\bigg(H_0+ \int_0^t(|\alpha(t')|+|\beta(t')|+|\gamma(t')|+|\xi(t')|)(H^{(n)}(t'))^2 dt'\bigg).
\end{split}
\end{equation}
Notice that the iterative inequality \eqref{3.5}  contains additional factor $|\alpha(t)|+|\gamma(t)|+|\gamma(t)|+|\xi(t)|$, the classic method used in \cite{danchin2001few} is inapplicable in present case. To overcome this difficulty, we use another iterative method to derive the uniform bound.

For $n=1$, it follows from \eqref{3.5} that
\begin{equation*}
\begin{split}
& \sup_{t\in [0,\infty)}H^{(1)}(t) \leq 2C e^{ 2C^2H_0\int_0^{\infty}(|\alpha(t')|+|\xi(t')|)dt'}\\
&  \quad \times\bigg(H_0+ 4C^2H_0^2 \int_0^{\infty}(|\alpha(t')|+|\beta(t')|+|\gamma(t')|+|\xi(t')|)dt'\bigg) = 2C h(H_0),
\end{split}
\end{equation*}
where
$$
h(x)\overset{\text{def}}{=}  e^{ 2C^2x\int_0^{\infty}(|\alpha(t')|+|\xi(t')|)dt'}\bigg(x+ 4C^2x^2 \int_0^{\infty}(|\alpha(t')|+|\beta(t')|+|\gamma(t')|+|\xi(t')|)dt'\bigg).
$$
It is clear that $h(0)=0$ and the function $h(x) $ is a modulus of continuity defined on $\mathbb{R}^+$, which is  independent of the initial data $(m_0,n_0)$.

For $k=2$, we deduce from \eqref{3.5}  that
\begin{equation}\label{3.6}
\begin{split}
\sup_{t\in [0,\infty)}H^{(2)}(t)\leq &Ce^{2C^2 h(H_0)\int_0^\infty(|\alpha(t')|+|\xi(t')|)dt'}\\
&  \times\bigg(H_0+ 4C^2 h^2(H_0)\int_0^\infty(|\alpha(t')|+|\beta(t')|+|\gamma(t')|+|\xi(t')|) dt'\bigg)\\
\leq &C h(H_0)e^{2C^2 h(H_0)\int_0^\infty(|\alpha(t')|+|\xi(t')|)dt'}\\
&  \times\bigg(1+ 4C^2 h(H_0)\int_0^\infty(|\alpha(t')|+|\beta(t')|+|\gamma(t')|+|\xi(t')|) dt'\bigg)\\
\leq &C h(H_0) e^{6C^2 h(H_0)\int_0^\infty(|\alpha(t')|+|\beta(t')|+|\gamma(t')|+|\xi(t')|) dt'}.
\end{split}
\end{equation}
In the last inequality of \eqref{3.6}, we have used the facts of $h(H_0)\geq H_0$ and $1+x \leq e^x$ for any $x\geq0$. To estimate the right hand side  of \eqref{3.6}, we assume that
\begin{equation*}
\begin{split}
&\int_0^\infty(|\alpha(t')|+|\beta(t')|+|\gamma(t')|+|\xi(t')|) dt'\\
&\quad \leq \frac{\ln 2}{6C^2 e^{ 2C^2H_0\int_0^{\infty}(|\alpha(t')|+|\xi(t')|)dt'}\left(H_0+ 4C^2H_0^2 \int_0^{\infty}(|\alpha(t')|+|\beta(t')|+|\gamma(t')|+|\xi(t')|)dt'\right)}\\
&\quad = \frac{\ln 2}{6C^2 h(H_0) },
\end{split}
\end{equation*}
which combined with \eqref{3.6} imply that
\begin{equation}\label{3.7}
\begin{split}
\sup_{t\in [0,\infty)}H^{(2)}(t)&\leq  C h(H_0) e^{6C^2 h(H_0)\int_0^\infty(|\alpha(t')|+|\beta(t')|+|\gamma(t')|+|\xi(t')|) dt'} \leq 2C h(H_0) .
\end{split}
\end{equation}
Assuming for any given $n\geq 3$ that
\begin{equation*}
\begin{split}
\sup_{t\in [0,\infty)}H^{(n)}(t) \leq 2Ch(H_0).
\end{split}
\end{equation*}
For $H^{(n+1)}(t)$, it follows from \eqref{3.5} and \eqref{3.7} that
\begin{eqnarray*}
\begin{split}
\sup_{t\in [0,\infty)}H^{(n+1)}(t) \leq &Ce^{C\int_0^\infty(|\alpha(t')|+|\xi(t')|)dt'\cdot\sup\limits_{t'\in [0,\infty)}H^{(n)}(t')}\\
&  \times\bigg(H_0+ \int_0^\infty(|\alpha(t')|+|\beta(t')|+|\gamma(t')|+|\xi(t')|)dt'\cdot\sup_{t'\in [0,\infty)}(H^{(n)}(t'))^2\bigg)\\
\leq &Ce^{2C^2h(H_0)\int_0^\infty(|\alpha(t')|+|\xi(t')|)dt'}\\
&  \times\bigg(H_0+ 4C^2h^2(H_0)\int_0^\infty(|\alpha(t')|+|\beta(t')|+|\gamma(t')|+|\xi(t')|)  dt'\bigg)\\
\leq &Ch(H_0)e^{2C^2h(H_0)\int_0^\infty(|\alpha(t')|+|\xi(t')|)dt'}\\
&  \times\bigg(1+ 4C^2h(H_0)\int_0^\infty(|\alpha(t')|+|\beta(t')|+|\gamma(t')|+|\xi(t')|)  dt'\bigg)\\
\leq &C h(H_0) e^{6C^2 h(H_0)\int_0^\infty(|\alpha(t')|+|\beta(t')|+|\gamma(t')|+|\xi(t')|) dt'}\\
\leq &2Ch(H_0).
\end{split}
\end{eqnarray*}
Using the mathematical induction with respect to $n$, we get
\begin{equation*}
\begin{split}
\sup_{n\geq 1}\sup_{t\in [0,\infty)}H^{(n+1)}(t) \leq  2Ch(H_0),
\end{split}
\end{equation*}
which implies the uniform bound
\begin{equation}\label{3.8}
\begin{split}
\sup_{n\geq 1}\Big(\|u^{(n)}\|_{L^\infty([0,\infty);B^{s}_{p,r})}+\|\sigma^{(n)}\|_{L^\infty([0,\infty);B^{s-1}_{p,r})}\Big)\leq 2Ch(H_0).
\end{split}
\end{equation}
As a result, the approximate solutions $(u^{(n)},\sigma^{(n)})_{n\geq1}$ is uniformly bounded in  $C([0,\infty); B^{s}_{p,r}\times B^{s-1}_{p,r})$. Moreover, using the system \eqref{3.1} itself, one can verify that the sequence $(\partial_tu^{(n)},\partial_t\sigma^{(n)})_{n\geq1}$ is uniformly bounded in $C([0,\infty); B^{s-1}_{p,r}\times B^{s-2}_{p,r})$. Therefore we obtain that $(u^{(n)},\sigma^{(n)})_{n\geq1}$ is uniformly bounded in $ E^{s}_{p,r}(\infty)$.

\vskip3mm\noindent
\textbf{Step 2 (Convergence):}
We are going to show that $(u^{(n)},\sigma^{(n)})_{n\geq1}$ is a Cauchy sequence in the space $C([0,\infty); B^{s-1}_{p,r}\times B^{s-2}_{p,r})$. Indeed, according to \eqref{3.1}, we obtain that for any $n,m\geq 1$
\begin{equation}\label{3.9}
\begin{split}
 &\partial_t (u^{(n+1)}-u^{(n+m+1)})+\alpha(t) u^{(n)}\partial_x (u^{(n+1)}-u^{(n+m+1)})\\
& \quad =-\alpha(t)( u^{(n)}-u^{(n+m)})u_x^{(n+m+1)} +\frac{\beta(t)}{2}\partial_xp*[(u^{(n+m)}+(u^{(n)})(u^{(n+m)}- u^{(n)})] \\
  & \quad \quad+\frac{\gamma(t)}{2}\partial_xp*[(\sigma^{(n+m)}+\sigma^{(n)})(\sigma^{(n+m)}- \sigma^{(n)})]\\
 &  \quad \quad +\frac{\beta(t)-3\alpha(t)}{2}\partial_xp*[(u_x^{(n+m)}+u_x^{(n)})(u^{(n+m)}_x- u^{(n)}_x)],
\end{split}
\end{equation}
and
\begin{equation}\label{3.10}
\begin{split}
&\partial_t(\sigma^{(n+1)}-^{(n+m+1)})+\xi(t)u^{(n)}\partial_x(\sigma^{(n+1)}-\sigma^{(n+m+1)})\\
& \quad =-\xi(t)( u^{(n)}-u^{(n+m)})\sigma_x^{(n+m+1)}+\xi(t)(\sigma^{(n+m)}- \sigma^{(n)}) u_x^{(n+m)}\\
 & \quad \quad +\xi(t)\sigma^{(n)} (u^{(n+m)}_x- u^{(n)}_x),
\end{split}
\end{equation}
with the initial conditions
\begin{equation*}
\begin{split}
& (u^{(n+1)}-u^{(n+m+1)})(0,x)=S_{n+1}u_0(x)-S_{n+m+1}u_0(x),\\
& (\sigma^{(n+1)}-\sigma^{(n+m+1)})(0,x)=S_{n+1}\sigma_0(x)-S_{n+m+1}\sigma_0(x).
\end{split}
\end{equation*}
Applying the Lemma \ref{lem:priorieti} to  \eqref{3.9}, we deduce that
\begin{equation}\label{3.11}
\begin{split}
&e^{-C\int_0^t\|\alpha(t')\partial_x u^{(n)}(t')\|_{B^{\frac{1}{p}}_{p,r}\cap L^\infty}dt'}\|(u^{(n+1)}-u^{(n+m+1)})(t)\|_{B^{s-1}_{p,r}} \\
&\quad \leq \|S_{n+1}u_0-S_{n+m+1}u_0\|_{B^{s-1}_{p,r}}+ \int_0^t e^{-C\int_0^{t'}\|\alpha(\tau)\partial_x u^{(n)}(\tau)\|_{B^{\frac{1}{p}}_{p,r}\cap L^\infty}d\tau}\\
&\quad \quad \times\bigg(\left\|\alpha(t')( u^{(n)}-u^{(n+m)})u_x^{(n+m+1)}\right\|_{B^{s-1}_{p,r}}+\left\|\frac{\beta(t')}{2}\partial_xp*[(u^{(n+m)}+u^{(n)})(u^{(n+m)}- u^{(n)})]\right\|_{B^{s-1}_{p,r}}\\
&\quad \quad    +\left\|\frac{\gamma(t')}{2}\partial_xp*[(\sigma^{(n+m)}+\sigma^{(n)})(\sigma^{(n+m)}- \sigma^{(n)})]\right\|_{B^{s-1}_{p,r}}\\
&\quad \quad   +\left\|\frac{\beta(t')-3\alpha(t')}{2}\partial_xp*[(u_x^{(n+m)}+u_x^{(n)})(u^{(n+m)}_x- u^{(n)}_x)]\right\|_{B^{s-1}_{p,r}}\bigg) dt' ,
\end{split}
\end{equation}
For $s>1+\frac{1}{p}$, the space $B^{s-1}_{p,r}$ is a Banach algebra and $B^{s-1}_{p,r}$ continuously embedded into $B^{\frac{1}{p}}_{p,r}\cap L^\infty$, we get from the uniform bound \eqref{3.8} that
\begin{equation*}
\begin{split}
\int_0^t\|\alpha(t')\partial_x u^{(n)}(t')\|_{B^{\frac{1}{p}}_{p,r}\cap L^\infty}dt'\leq C\int_0^t|\alpha(t')|\|u^{(n)}(t')\|_{B^{s}_{p,r}}dt'\leq C\int_0^t|\alpha(t')|dt',
\end{split}
\end{equation*}
\begin{equation*}
\begin{split}
\left\|\alpha(t')( u^{(n)}-u^{(n+m)})u_x^{(n+m+1)}\right\|_{B^{s-1}_{p,r}}&\leq C|\alpha(t)| \|u^{(n+m+1)} \|_{B^{s }_{p,r}}\|u^{(n)}-u^{(n+m)}\|_{B^{s-1 }_{p,r}}\\
&\leq C|\alpha(t)| \|u^{(n)}-u^{(n+m)}\|_{B^{s-1 }_{p,r}},
\end{split}
\end{equation*}
and
\begin{equation*}
\begin{split}
&\left\|\frac{\beta(t')}{2}\partial_xp*[(u^{(n+m)}+u^{(n)})(u^{(n+m)}- u^{(n)})]\right\|_{B^{s-1}_{p,r}}\\
&\quad \leq C |\beta(t')|(\|u^{(n+m)}\|_{B^{s}_{p,r}}+\|u^{(n)}\|_{B^{s}_{p,r}})\|u^{(n+m)}- u^{(n)}\|_{B^{s-1}_{p,r}}\\
&\quad \leq C |\beta(t')|\|u^{(n+m)}- u^{(n)}\|_{B^{s-1}_{p,r}},
\end{split}
\end{equation*}

$\bullet$ For $\max\{1+\frac{1}{p},\frac{3}{2}\}< s\leq 2+\frac{1}{p}$, we first recall the following conclusion.
\begin{lemma}[Moser estimate \cite{danchin2001few}] \label{moser}
Let
$s_1\leq 1/p<s_2$ $(s_2\geq 1/p$ if $r=1$) and $s_1+s_2>0$, then
\begin{equation*}
\begin{split}
\|fg\|_{B^{s_1}_{p,r}}\leq  C\|f\|_{B^{s_1}_{p,r}}\|g\|_{B^{s_2}_{p,r}}.
\end{split}
\end{equation*}
\end{lemma}
If we take
$$
s_1=s-2 \leq \frac{1}{p}< s_2=s-1,
$$
then the last two terms in \eqref{3.11} can be estimated as follows
\begin{equation*}
\begin{split}
&\left\|\frac{\gamma(t')}{2}\partial_xp*[(\sigma^{(n+m)}+\sigma^{(n)})(\sigma^{(n+m)}- \sigma^{(n)})]\right\|_{B^{s-1}_{p,r}}\\
&\quad \leq C|\gamma(t')|(\|\sigma^{(n+m)}\|_{B^{s-1}_{p,r}}+\|\sigma^{(n)}\|_{B^{s-1}_{p,r}})\|\sigma^{(n+m)}- \sigma^{(n)}\|_{B^{s-2}_{p,r}}\\
&\quad \leq C|\gamma(t')|\|\sigma^{(n+m)}- \sigma^{(n)}\|_{B^{s-2}_{p,r}},
\end{split}
\end{equation*}
and
\begin{equation*}
\begin{split}
&\left\|\frac{\beta(t')-3\alpha(t')}{2}\partial_xp*[(u_x^{(n+m)}+u_x^{(n)})(u^{(n+m)}_x- u^{(n)}_x)]\right\|_{B^{s-1}_{p,r}}\\
&\quad \leq \frac{|\beta(t')-3\alpha(t')|}{2}(\|u^{(n+m)}\|_{B^{s}_{p,r}}+\|u^{(n)}\|_{B^{s}_{p,r}})\|u^{(n+m)}_x- u^{(n)}_x\|_{B^{s-2}_{p,r}}\\
&\quad \leq C |\beta(t')-3\alpha(t')| \|u^{(n+m)}- u^{(n)}\|_{B^{s-1}_{p,r}}.
\end{split}
\end{equation*}

$\bullet$ If $s >2+\frac{1}{p}$, the space $B^{s-2}_{p,r}$ is a Banach algebra, then it is not difficult to obtain the similar estimates by using the uniform bound  \eqref{3.8}.

As a consequence, we deduce from  \eqref{3.11} that
\begin{equation}\label{3.12}
\begin{split}
&\|(u^{(n+1)}-u^{(n+m+1)})(t)\|_{B^{s-1}_{p,r}} \\
&\quad \leq e^{C\int_0^t|\alpha(t')|dt'}\bigg(\|S_{n+1}u_0-S_{n+m+1}u_0\|_{B^{s-1}_{p,r}}  \\
&\quad \quad  + \int_0^t\Big((|\alpha(t')|+|\beta(t')|)\|u^{(n+m)}- u^{(n)}\|_{B^{s-1}_{p,r}} +|\gamma(t')|\|\sigma^{(n+m)}- \sigma^{(n)}\|_{B^{s-2}_{p,r}}\Big) dt' \bigg).
\end{split}
\end{equation}

Next, we apply the Lemma \ref{lem:priorieti} to  \eqref{3.10} to find
\begin{equation}\label{3.13}
\begin{split}
&\|(\sigma^{(n+1)}-\sigma ^{(n+m+1)})(t)\|_{B^{s-2}_{p,r}} \leq e^{C\int_0^t|\xi(t')|\|u^{(n)}\|_{B^{s}_{p,r}}dt'}\bigg(\|S_{n+1}\sigma_0-S_{n+m+1}\sigma_0\|_{B^{s-2}_{p,r}}\\
&\quad + \int_0^t |\xi(t')|\Big(\| u^{(n)}-u^{(n+m)}\|_{B^{s-1}_{p,r}}\|\sigma_x^{(n+m+1)}\|_{B^{s-2}_{p,r}}\\
&\quad+\|\sigma^{(n+m)}- \sigma^{(n)}\|_{B^{s-2}_{p,r}}\| u_x^{(n+m)}\|_{B^{s-1}_{p,r}} +\|\sigma^{(n)}\|_{B^{s-1}_{p,r}}\| u^{(n+m)}_x- u^{(n)}_x\|_{B^{s-2}_{p,r}}\Big)\bigg)\\
&\leq e^{C\int_0^t|\xi(t')dt'}\bigg(\|S_{n+1}\sigma_0-S_{n+m+1}\sigma_0\|_{B^{s-2}_{p,r}}\\
&\quad + \int_0^t |\xi(t')|\Big(\| u^{(n)}-u^{(n+m)}\|_{B^{s-1}_{p,r}} +\|\sigma^{(n+m)}- \sigma^{(n)}\|_{B^{s}_{p,r}}\Big)dt'\bigg)
\end{split}
\end{equation}
According to the definition of the Littlewood-Paley blocks $\Delta_p$ and the almost orthogonal property $\Delta_i\Delta_j=0$ for $|i-j|\geq 2$, we have
\begin{equation}\label{3.14}
\begin{split}
 \|S_{n+1}u_0-S_{n+m+1}u_0\|_{B^{s-1}_{p,r}}^r& =\sum_{j\geq -1}2^{jr(s-1)}\left\|\Delta_j\sum_{n+1\leq i \leq n+m}\Delta_i u_0\right\|_{L^p}^r\\
 &\leq C\sum_{n\leq j \leq n+m+1}\left(2^{-jr}2^{jrs}\sum_{n+1\leq i \leq n+m}\left\|\Delta_j \Delta_iu_0\right\|_{L^2}^r\right)\\
 &\leq C2^{-rn}\sum_{n\leq j \leq n+m+1}2^{jrs}\left\|\Delta_j u_0\right\|_{L^2}^r\leq C2^{-rn} \|u_0\|_{B_{p,r}^{s}}^r,
\end{split}
\end{equation}
and similarly,
\begin{equation}\label{3.15}
\begin{split}
\|S_{n+1}\sigma_0-S_{n+m+1}\sigma_0\|_{B^{s-2}_{p,r}}^r\leq C2^{-rn} \|\sigma_0\|_{B_{p,r}^{s-1}}^r.
\end{split}
\end{equation}
For simplicity, we set
$$
\mathcal {H}^{(n,m)}(t)\overset{\text{def}}{=} \|(u^{(n)}-u^{(n+m)})(t)\|_{B^{s-1}_{p,r}}+\|(\sigma^{(n)}-\sigma ^{(n+m)})(t)\|_{B^{s-2}_{p,r}}.
$$
Putting the estimates  \eqref{3.12}- \eqref{3.15} together and using the inequality $e^x+e^y\leq 2 e^{x+y}$ for any $x,y\geq 0$,  we obtain
\begin{equation}
\begin{split}
\mathcal {H}^{(n+1,m)}(t) \leq & C\bigg(2^{-n} + \int_0^t(|\alpha(t')|+|\beta(t')|+|\gamma(t')|+|\xi(t')|)\mathcal {H}^{(n,m)}(t')dt'\bigg).
\end{split}
\end{equation}
To derive an estimate for $\mathcal {H}^{(n+1,m)}(t)$, we first prove the following useful lemma.
\begin{lemma} \label{lem:3.1}
Let $\{a_n\}$ be a positive constant, and the nonnegative function $\mu(t)\in L^1(0,\infty)$. Assume that the sequence of nonnegative functions $(g_n(t))_{n\geq1}$ satisfies the inequality
$$
g_{n+1}(t)\leq a_n+\int_0^t \mu(t') g_n(t')dt',
$$
where $g_0$ is a nonnegative number. Then we have
$$
\sup_{t\in [0,\infty)}g_{n+1}(t) \leq \sum_{k=0}^{n}\frac{a_{n-k}\|\mu\|_{L^1(0,\infty)}^k }{k!}+ \frac{g_{0}\|\mu\|_{L^1(0,\infty)}^{n+1} }{(n+1)!}.
$$
\end{lemma}

\begin{proof} Using the iterative inequality, we first have
\begin{equation}\label{3.17}
\begin{split}
g_{n+1}(t)&\leq a_n+\int_0^t \mu(t_2) \left(a_{n-1}+\int_0^{t_2} \mu(t_1) g_{n-1}(t_1)dt_1\right)dt_2\\
&\leq a_n+a_{n-1}\int_0^t \mu(t_2)dt_2+ \int_0^t \int_0^{t_2} \mu(t_1) g_{n-1}(t_1)\mu(t_2)dt_1dt_2\\
&= a_n+a_{n-1}\int_0^t \mu(t_2)dt_2+ \int_0^t \mu(t_1)\left(\int^t_{t_1}\mu(t_2)dt_2 \right)g_{n-1}(t_1)dt_1.
\end{split}
\end{equation}
Inserting the inequality  $g_{n-1}(t_1)\leq a_{n-2}+\int_0^{t_1} \mu(t_3) g_{n-2}(t_3)dt_3$ into  \eqref{3.17}, we have
\begin{equation}\label{3.18}
\begin{split}
g_{n+1}(t) &\leq  a_n+a_{n-1}\int_0^t \mu(t_2)dt_2+ a_{n-2}\int_0^t \mu(t_1)\left(\int^t_{t_1}\mu(t_2)dt_2 \right)dt_1\\
&\quad + \int_0^t \mu(t_1)\left(\int^t_{t_1}\mu(t_2)dt_2 \right)\left(\int_0^{t_1} \mu(t_3) g_{n-2}(t_3)dt_3\right)dt_1.
\end{split}
\end{equation}
For the last two terms on the right hand side of  \eqref{3.18}, we have
\begin{equation*}
\begin{split}
\int_0^t \mu(t_1)\left(\int^t_{t_1}\mu(t_2)dt_2 \right)dt_1 =\frac{1}{2}\left(\int^t_{0}\mu(t_2)dt_2\right)^2,
\end{split}
\end{equation*}
and
\begin{equation*}
\begin{split}
&\int_0^t \mu(t_1)\left(\int^t_{t_1}\mu(t_2)dt_2 \right)\left(\int_0^{t_1} \mu(t_3) g_{n-2}(t_3)dt_3\right)dt_1\\
&\quad=-\frac{1}{2}\int_0^t \left(\int_0^{t_1} \mu(t_3) g_{n-2}(t_3)dt_3\right)d_{t_1}\left[\left(\int^t_{t_1}\mu(t_2)dt_2\right)^2\right]\\
&\quad=\frac{1}{2}\int_0^t \mu(t_1)\left(\int^t_{t_1}\mu(t_2)dt_2\right)^2 g_{n-2}(t_1)dt_1.
\end{split}
\end{equation*}
Hence we get
\begin{equation}\label{3.19}
\begin{split}
g_{n+1}(t)\leq&  a_n+a_{n-1}\int_0^t \mu(t_2)dt_2+ \frac{a_{n-2}}{2}\left(\int^t_{0}\mu(t_2)dt_2\right)^2\\
&+ \frac{1}{2}\int_0^t \mu(t_1)\left(\int^t_{t_1}\mu(t_2)dt_2\right)^2 g_{n-2}(t_1)dt_1.
\end{split}
\end{equation}
For the integrand  $g_{n-2}(t_1)$ in  \eqref{3.19}, we use the iterative inequality again and obtain
\begin{equation}\label{3.20}
\begin{split}
g_{n+1}(t)\leq&  a_n+a_{n-1}\int_0^t \mu(t_2)dt_2+ \frac{a_{n-2}}{2}\left(\int^t_{0}\mu(t_2)dt_2\right)^2+ \frac{a_{n-3}}{2}\int_0^t \mu(t_1)\left(\int^t_{t_1}\mu(t_2)dt_2\right)^2 dt_1\\
&+ \frac{1}{2}\int_0^t \mu(t_1)\left(\int^t_{t_1}\mu(t_2)dt_2\right)^2 \left(\int_0^{t_1} \mu(t_4) g_{n-3}(t_4)dt_4\right)dt_1\\
=&  \sum_{k=0}^3\frac{a_{n-k}}{k!}\left(\int^t_{0}\mu(t_2)dt_2\right)^k + \frac{1}{3!}\int_0^t \mu(t_1)\left(\int^t_{t_1}\mu(t_2)dt_2\right)^3  g_{n-3}(t_1)  dt_1.
\end{split}
\end{equation}
Following the similar procedure for several times, we finally obtain
\begin{equation*}
\begin{split}
g_{n+1}(t)&\leq \sum_{k=0}^{n}\frac{a_{n-k}}{k!}\left(\int^t_{0}\mu(t_2)dt_2\right)^k + \frac{g_{0} }{n!}\int_0^t \mu(t_1)\left(\int^t_{t_1}\mu(t_2)dt_2\right)^{n}  dt_1\\
&=\sum_{k=0}^{n}\frac{a_{n-k}}{k!}\left(\int^t_{0}\mu(t_2)dt_2\right)^k + \frac{g_{0} }{(n+1)!}\left(\int^t_{0}\mu(t_2)dt_2\right)^{n+1}\\
&\leq \sum_{k=0}^{n}\frac{a_{n-k}\|\mu\|_{L^1(0,\infty)}^k }{k!}+ \frac{g_{0}\|\mu\|_{L^1(0,\infty)}^{n+1} }{(n+1)!}.
\end{split}
\end{equation*}
The proof of Lemma \ref{lem:3.1} is completed.
\end{proof}

By taking $a_n=C2^{-n}$ and $\mu(t)=C(|\alpha(t)|+|\beta(t)|+|\gamma(t)|+|\xi(t)|)$, we get from Lemma \ref{lem:3.1} that
\begin{equation*}
\begin{split}
\sup_{t\in [0,\infty)}\mathcal {H}^{(n+1,m)}(t) \leq \frac{1}{2^n}\sum_{k=0}^{n}\frac{ C}{k!}\left(\frac{\ln 2}{3C^2 h(H_0) }\right)^k+ \frac{g_{0} }{(n+1)!}\left(\frac{\ln 2}{6C^2 h(H_0) }\right)^{n+1},
\end{split}
\end{equation*}
which implies that
\begin{equation*}
\begin{split}
\lim_{n\rightarrow\infty}\sup_{t\in [0,\infty)}\mathcal {H}^{(n+1,m)}(t)=0.
\end{split}
\end{equation*}
Therefore, the sequence $(u^{(n)},\sigma^{(n)})_{n\geq1}$ converges strongly in the Banach space $ C([0,\infty); B^{s-1}_{p,r}\times B^{s-2}_{p,r})$, and we denote the limit function  by $(u,\sigma)$.

\vskip3mm\noindent
\textbf{Step 3 (Existence):}  In this step we shall verify that the limit function $(u,\sigma)$ indeed belongs to $E_{p,r}^s(\infty)$ and is a strong solution to the system \eqref{1.6}. Since the sequence $(u^{(n)},\sigma^{(n)})_{n\geq1}$ is uniformly bounded in $L^\infty([0,\infty); B^{s}_{p,r}\times B^{s-1}_{p,r})$, and $(u^{(n)},\sigma^{(n)})\rightarrow (u,\sigma)$ strongly in $B^{s-1}_{p,r}\times B^{s-2}_{p,r}\hookrightarrow \mathcal {S}'\times \mathcal {S}'$ as $n\rightarrow\infty$, it follows from the Fatou's lemma (cf. \cite{64}) that $(u,\sigma)\in L^\infty([0,\infty); B^{s}_{p,r}\times B^{s-1}_{p,r})$.

On the other hand, as $(u^{(n)},\sigma^{(n)})_{n\geq1}$ converges strongly to $(u,\sigma)$ in $ C([0,\infty); B^{s-1}_{p,r}\times B^{s-2}_{p,r})$, an interpolation argument insures that the convergence holds in $ C([0,\infty); B^{s'}_{p,r}\times B^{s'-1}_{p,r})$, for any $s'<s$. It is then easy to pass to the limit in the system \eqref{3.1} and to conclude that $(u,\sigma)$ is indeed a solution to \eqref{1.6}.  Thanks to the fact that $(u, \sigma)$ belongs to
$C([0,\infty); B^{s}_{p,r}\times B^{s-1}_{p,r})$, we know that the right-hand side of the equation
\begin{equation*}
 u_t+\alpha(t) uu_x=-\partial_xp*\left(\frac{\beta(t)}{2} u^2+\frac{\gamma(t)}{2} \sigma^2+\frac{\beta(t)-3\alpha(t)}{2}u_x^2\right),
\end{equation*}
belongs to $L^1([0,\infty); B^{s}_{p,r})$, and the right-hand side of the equation
\begin{equation*}
\sigma_t+\xi(t)u\sigma_x=-\xi(t)\sigma u_x,
\end{equation*}
 belongs to $L^1([0,\infty); B^{s-1}_{p,r})$. In particular, for the case $r < \infty$, applying Lemma \ref{lem:transwell} implies that $ (u,\sigma)\in C([0,\infty); B^{s'}_{p,r}\times B^{s'-1}_{p,r})$, for any $s'<s$. Finally, using the system \eqref{1.6} again, we see that $ (\partial_tu,\partial_t\sigma)\in C([0,\infty); B^{s'-1}_{p,r}\times B^{s'-2}_{p,r})$ for $r < \infty$, and in $L^\infty([0,\infty); B^{s-1}_{p,r}\times B^{s-2}_{p,r})$ otherwise. Therefore, $(u, \sigma)$ belongs to $E_{p,r}^s(\infty)$. Moreover, a standard use of a sequence of viscosity approximate solutions $(u_\epsilon, \sigma_\epsilon )_{\epsilon>0}$ for system \eqref{1.6} which converges uniformly in
$$
C([0,\infty); B^{s}_{p,r}\times B^{s-1}_{p,r})\cap C^1([0,\infty); B^{s-1}_{p,r}\times B^{s-2}_{p,r})
$$
gives the continuity of solution $(u, \sigma)$ in  $E_{p,r}^s(\infty)$.

\vskip3mm\noindent
\textbf{Step 4 (Uniqueness and stability):} Let $(u^{(i)},\sigma^{(i)})$ is the solution to system \eqref{1.6} with respect to the initial data $(u^{(i)}_0,\sigma^{(i)}_0)$, $i=1,2$. Setting
$$
u^{(12)}= u^{(1)}-u^{(2)},\quad \sigma^{(12)}=\sigma^{(1)}-\sigma^{(2)}.
$$
Then the pair $(u^{(12)},\sigma^{(12)})$ satisfies the following system
\begin{equation} \label{3.22}
\begin{cases}
 \partial_t u^{(12)}+\alpha(t) u^{(1)}\partial_x u^{(12)}=-\alpha(t)u^{(12)}u_x^{(2)} +f(u^{(12)},u^{(12)}, u^{(1)}+ u^{(2)}, \sigma^{(1)}+ \sigma^{(2)}),\\
\partial_t\sigma^{(12)}+\xi(t)u^{(1)}\partial_x\sigma^{(12)}=-\xi(t)u^{(12)}\sigma_x^{(2)}-\xi(t)\sigma^{(12)}u_x^{(2)} -\xi(t)\sigma^{(1)}u^{(12)}_x,\\
u^{(12)}(0,x)\overset{\text{def}}{=} u^{(12)}_0=u^{(1)}_0-u^{(2)}_0,\\
\sigma^{(12)}(0,x)\overset{\text{def}}{=} \sigma^{(12)}_0=\sigma^{(1)}_0-\sigma^{(2)}_0,
\end{cases}
\end{equation}
where
\begin{equation*}
\begin{split}
&f(u^{(12)},u^{(12)}, u^{(1)}+ u^{(2)}, \sigma^{(1)}+ \sigma^{(2)})\\
&\quad \overset{\text{def}}{=}-\frac{\beta(t)}{2}\partial_xp*[u^{(12)}(u^{(2)}+u^{(1)})]-\frac{\gamma(t)}{2}\partial_xp*[\sigma^{(12)}(\sigma^{(2)}+\sigma^{(1)})] \\
&\quad \quad-\frac{\beta(t)-3\alpha(t)}{2}\partial_xp*[u^{(12)}_x(u_x^{(2)}+u_x^{(1)})].
\end{split}
\end{equation*}
According to Lemma \ref{lem:priorieti}, we have the following estimates for \eqref{3.22}:
\begin{equation}\label{3.23}
\begin{split}
&e^{-C\int_0^t|\alpha(t')|\|u^{(1)}(t')\|_{B^{s}_{p,r}}dt'}|u^{(12)}(t)\|_{B^{s-1}_{p,r}} \\
&\quad  \leq\|u^{(12)}_0\|_{B^{s-1}_{p,r}}+ \int_0^t e^{-C\int_0^{t'}|\alpha(\tau)|\|u^{(1)}(\tau)\|_{B^{s}_{p,r}}d\tau} \\
&\quad \quad\times\bigg(\|\alpha(t')u^{(12)}u_x^{(2)}\|_{B^{s-1}_{p,r}} +\|f(u^{(12)},u^{(12)}, u^{(1)}+ u^{(2)}, \sigma^{(1)}+ \sigma^{(2)})\|_{B^{s-1}_{p,r}} \bigg)dt' ,
\end{split}
\end{equation}
and
\begin{equation}\label{3.24}
\begin{split}
&e^{-C\int_0^t|\xi(t')|\|u^{(1)}(t')\|_{B^{s}_{p,r}}dt'}\|\sigma^{(12)}(t)\|_{B^{s-2}_{p,r}}  \\
&\quad \leq \|\sigma^{(12)}_0\|_{B^{s-2}_{p,r}}+ \int_0^t e^{-C\int_0^{t'}|\xi(\tau)| \|u^{(1)}(\tau)\|_{B^{s}_{p,r}}d\tau}\\
&\quad \quad \times\bigg(\|\xi(t')u^{(12)}\sigma_x^{(2)}\|_{B^{s-2}_{p,r}}+\|\xi(t')\sigma^{(12)}u_x^{(2)}\|_{B^{s-2}_{p,r}} +\|\xi(t')\sigma^{(1)}u^{(12)}_x\|_{B^{s-2}_{p,r}} \bigg)dt' .
\end{split}
\end{equation}
Since $B^{s-1}_{p,r}$ is a Banach algebra for any $s>1+\frac{1}{p}$, so we have
\begin{equation*}
\begin{split}
\|\alpha(t')u^{(12)}u_x^{(n+m+1)}\|_{B^{s-1}_{p,r}}&\leq C| \alpha(t')|\|u_x^{(n+m+1)}\|_{B^{s-1}_{p,r}}\|u^{(12)}\|_{B^{s-1}_{p,r}} ,\\
\|\frac{\beta(t')}{2}\partial_xp*[(u^{(n+m)}+u^{(n)})u^{(12)}]\|_{B^{s-1}_{p,r}} &\leq C|\beta(t')| \|u^{(n+m)}+u^{(n)}\|_{B^{s}_{p,r}} \|u^{(12)} \|_{B^{s-1}_{p,r}}, \\
\|\frac{\beta(t')-3\alpha(t')}{2}\partial_xp*[(u_x^{(n+m)}+u_x^{(n)})u^{(12)}_x]\|_{B^{s-1}_{p,r}}&\leq C|\beta(t')-3\alpha(t')|\|u_x^{(n+m)}+u_x^{(n)})\|_{B^{s-1}_{p,r}}\|u^{(12)}_x\|_{B^{s-2}_{p,r}},\\
\|\frac{\gamma(t')}{2}\partial_xp*[(\sigma^{(n+m)}+\sigma^{(n)})\sigma^{(12)}] \|_{B^{s-1}_{p,r}}&\leq C|\gamma(t')|\|\sigma^{(n+m)}+\sigma^{(n)}\|_{B^{s-1}_{p,r}}\|\sigma^{(12)}\|_{B^{s-2}_{p,r}}.
\end{split}
\end{equation*}
Therefore, one can deduce that
\begin{equation*}
\begin{split}
&\|f(u^{(12)},u^{(12)}, u^{(1)}+ u^{(2)}, \sigma^{(1)}+ \sigma^{(2)})\|_{B^{s-1}_{p,r}}\\
&\quad\leq C  (|\alpha(t')|+|\beta(t')|+|\gamma(t')|) (\|u^{(1)}\|_{B^{s-1}_{p,r}}+\|u^{(2)}\|_{B^{s-1}_{p,r}}+\|\sigma^{(1)}\|_{B^{s-1}_{p,r}}+\|\sigma^{(2)}\|_{B^{s-1}_{p,r}})\\
&\quad\quad\times(\|u^{(12)}\|_{B^{s-1}_{p,r}}+ \|\sigma^{(12)}\|_{B^{s-2}_{p,r}})
\end{split}
\end{equation*}
For $\max\{1+\frac{1}{p},\frac{3}{2}\}<s\leq 2+\frac{1}{p}$, we get by using the Lemma \ref{lem:priorieti}
\begin{equation*}
\begin{split}
 \|\xi(t')u^{(12)}\sigma_x^{(2)}\|_{B^{s-2}_{p,r}} &\leq C| \xi(t')|\|u^{(12)}\|_{B^{s-2}_{p,r}}\|\sigma_x^{(2)}\|_{B^{s-2}_{p,r}} ,\\
 \|\xi(t')\sigma^{(12)}u_x^{(2)}\|_{B^{s-2}_{p,r}}&\leq C| \xi(t')|\|\sigma^{(12)}\|_{B^{s-2}_{p,r}}\|u_x^{(2)}\|_{B^{s-1}_{p,r}},\\
 \|\xi(t')\sigma^{(1)}u^{(12)}_x\|_{B^{s-2}_{p,r}}&\leq C| \xi(t')|\|\sigma^{(1)}\|_{B^{s-1}_{p,r}}\|u^{(12)}_x\|_{B^{s-2}_{p,r}}.
\end{split}
\end{equation*}
Similar estimates also hold for the case of $s> 2+\frac{1}{p}$. Putting above estimates together, we get from \eqref{3.23} and \eqref{3.24} that
\begin{equation*}
\begin{split}
&e^{-C\int_0^t(|\alpha(t')|+|\xi(t')|)\|u^{(1)}(t')\|_{B^{s}_{p,r}}dt'}(\|u^{(12)}(t)\|_{B^{s-1}_{p,r}} +\|\sigma^{(12)}(t)\|_{B^{s-2}_{p,r}})\\
&\quad  \leq \|u^{(12)}_0\|_{B^{s-1}_{p,r}}+|\sigma^{(12)}_0\|_{B^{s-2}_{p,r}}+ \int_0^t e^{-C\int_0^{t'}(|\alpha(\tau)|+|\xi(\tau)|)\|u^{(1)}(\tau)\|_{B^{s}_{p,r}}d\tau} \\
&\quad \quad\times(\|u^{(12)}(t')\|_{B^{s-1}_{p,r}}+ \|\sigma^{(12)}(t')\|_{B^{s-2}_{p,r}})(|\alpha(t')|+|\beta(t')|+|\gamma(t')|+|\xi(t')|)\\
&\quad \quad\times(\|u^{(1)}\|_{B^{s-1}_{p,r}}+\|u^{(2)}\|_{B^{s-1}_{p,r}}+\|\sigma^{(1)}\|_{B^{s-1}_{p,r}}+\|\sigma^{(2)}\|_{B^{s-1}_{p,r}})dt' ,
\end{split}
\end{equation*}
Applying the Gronwall lemma to above inequality leads to
\begin{equation*}
\begin{split}
&\|u^{(12)}(t)\|_{B^{s-1}_{p,r}} +\|\sigma^{(12)}(t)\|_{B^{s-2}_{p,r}}\leq (\|u^{(12)}_0\|_{B^{s-1}_{p,r}}+|\sigma^{(12)}_0\|_{B^{s-2}_{p,r}})\\
&\quad \times e^{C\int_0^t(|\alpha(t')|+|\beta(t')|+|\gamma(t')|+|\xi(t')|)(\|u^{(1)}\|_{B^{s-1}_{p,r}}+\|u^{(2)}\|_{B^{s-1}_{p,r}}+\|\sigma^{(1)}\|_{B^{s-1}_{p,r}}+\|\sigma^{(2)}\|_{B^{s-1}_{p,r}})dt'},
\end{split}
\end{equation*}
which implies the uniqueness and continuity with respect to the initial data. The proof of Theorem \ref{theorem1} is now completed.

\section{Blow-up phenomena}
In this section, we derive a precise blow-up criteria of strong solutions to the system \eqref{1.6} (when $\beta(t)=3\alpha(t)$) with initial data in Sobolev spaces, and gives specific characterization for the lower bound of the blow-up time. Additionally, with sufficient conditions on the time-dependent parameters $\alpha,\gamma$ and $\xi$, we provide a precise wave-breaking criteria for the strong solution, which is shown to be independent of the second component.

In the case of $\beta(t)=3\alpha(t)$, the system \eqref{1.5} reduces to the following form:
\begin{equation} \label{4.9}
\begin{cases}
 m_t+\alpha(t)um_x+3\alpha(t)u_xm +\gamma(t)\sigma\sigma_x=0,& t>0,~x\in \mathbb{K},\\
 \sigma_t+\xi(t)(u\sigma)_x=0,& t>0,~x\in \mathbb{K},\\
 m(0,x)=m_0(x),\quad
 \sigma(0,x)=\sigma(x),&  x\in \mathbb{K},
\end{cases}
\end{equation}
If the parameters $\alpha,\gamma$ an $\xi$ are merely locally integrable, then Theorem \ref{theorem1} implies that, given $(u_0,\sigma_0)\in H^s\times H^{s-1}$ ($s>\frac{3}{2}$), the system \eqref{4.9} has  a unique  local-in-time solution $(u,\sigma)\in C([0,T],H^s\times H^{s-1})$ for some $T>0$.

Now we consider the flow $t\rightarrow \psi (t,x)$ generated by $\xi(t)u(t,x)$:
\begin{equation*}
\begin{cases}
 \frac{d \psi(t,x)}{dt}=\xi(t)u(t,\psi(t,x)),\\
\psi(0,x)=x,
\end{cases} t\geq0,~~x\in \mathbb{K}.
\end{equation*}
By the regularity of $u(t,x)$ in $x$-variable and the classic ODE theory,  the previous initial value problem has a unique solution $ \psi\in C^1([0,T);\mathbb{R})$. Moreover, direct calculation leads to
\begin{equation*}
\begin{cases}
 \frac{d \psi_x(t,x)}{dt}=\xi(t)u_x(t,\psi(t,x))\psi_x(t,x),\\
\psi_x(0,x)=1,
\end{cases} t\geq0,~~x\in \mathbb{K}.
\end{equation*}
Hence for any $t\in [0,T)$, $x \in \mathbb{K}$, we find that
\begin{equation}
\psi_x(t,x)=e^{ \int_0^t\xi(t')u_x(t',\psi(t',x))dt'}>0,
\end{equation}
which implies that $\psi(t,\cdot): \mathbb{K}\rightarrow \mathbb{K}$ is a diffeomorphism on $\mathbb{K}$ for every $t\in [0,T)$.

Using the second component of \eqref{1.6}, it is easy to find that
\begin{equation*}
\begin{split}
&\frac{d}{dt}( \sigma(t,\psi(t,x))\psi_x(t,x))\\
&\quad =\sigma_t(t,\psi(t,x))\psi_x(t,x)+\sigma_x(t,\psi(t,x))\psi_t(t,x)\psi_x(t,x)+\sigma(t,\psi(t,x))\psi_{xt}(t,x)\\
&\quad =[\sigma_t +\xi(t)\sigma_x u + \xi(t)\sigma u_x ](t,\psi(t,x)) \psi_{x}(t,x) =0,
\end{split}
\end{equation*}
which implies
\begin{equation*}
\begin{split}
 \sigma(t,\psi(t,x))\psi_x(t,x)=\sigma(0,\psi(0,x))\psi_x(0,x)=\sigma_0(x).
\end{split}
\end{equation*}
As a result, we have the following $L^\infty$-bound
\begin{equation} \label{bound}
\begin{split}
 \|\sigma(t,\cdot)\|_{L^\infty}=& \|\sigma(t,\psi(t,\cdot))\|_{L^\infty}\\
 \leq& \|\sigma_0 \|_{L^\infty} \left\|e^{-\int_0^t\xi(t')u_x(t',\psi(t',\cdot))dt'}\right\|_{L^\infty}\\
 \leq& \|\sigma_0 \|_{L^\infty} e^{-\int_0^t\inf_{x\in \mathbb{K}} \{\xi(t') u_x(t',x)\}dt'}.
\end{split}
\end{equation}
and the $L^2$-bound
\begin{equation} \label{bound1}
\begin{split}
 \|\sigma(t,\cdot)\|_{L^2} =&\left(\int_ \mathbb{K} |\sigma(t,\psi(t,x))|^2\psi_x(t,x)dx\right)^{ \frac{1}{2}}\\
 =&\left(\int_ \mathbb{K}  |\sigma_0(x)|^2\psi_x(t,x)^{-1}dx\right)^{ \frac{1}{2}} \leq \|\sigma_0\|_{L^2} e^{-\frac{1}{2}\int_0^t\inf_{x\in \mathbb{K}} \{\xi(t')u_x(t',x)\}  dt'}.
\end{split}
\end{equation}

The following lemma provides the $L^\infty$-bound and $L^2$-bound for the first component $u$ of the solution, which is crucial in the proof of the following blow-up criteria.
\begin{lemma}
Assume that the parameters $\alpha,\gamma,\xi\in L^1_{loc}([0,\infty);\mathbb{R})$. Let $(u,\sigma)$ be the solution to the system (4.1) with initial data $(u_0,\sigma_0)\in H^s\times H^{s-1}$ with $s>\frac{3}{2}$, then we have the $L^2$-bound
\begin{equation*}
\begin{split}
\| u(t)\|_{L^2} \leq2\| u_0\|_{L^2}\exp\left\{2\|\sigma_0 \|_{H^{s-1}} ^4 \|\gamma\|_{L^1(0,t)}e^{-3\int_0^t\inf_{x\in \mathbb{K}} \{\xi(t') u_x(t',x)\}dt'}\right\},
\end{split}
\end{equation*}
and the $L^\infty$-bound
\begin{equation*}
\begin{split}
  \|u(t)\|_{L^\infty} \leq \|u_0\|_{L^\infty}+t\mathcal {J}(t ) ,
\end{split}
\end{equation*}
where
\begin{equation*}
\begin{split}
\mathcal {J}(t) &\overset{\text{def}}{=}4\| u_0\|_{L^2}^2|\alpha(t)|\exp\left\{4\|\sigma_0 \|_{H^{s-1}} ^4\|\gamma\|_{L^1(0,t)}e^{-3\int_0^t\inf_{x\in \mathbb{K}} \{\xi(t') u_x(t',x)\}dt'}\right\}\\
&\quad +\|\sigma_0\|_{L^2}^2|\gamma(t)| e^{- \int_0^t\inf_{x\in \mathbb{K}} \{\xi(t')u_x(t',x)\}  dt'} .
\end{split}
\end{equation*}
\end{lemma}

\begin{proof}
By a standard density argument, it is sufficient to prove the lemma by assuming $s\geq3$. Define $\chi= (4-\partial_x^2)^{-1}u$, then we have $m=(1-\partial_x^2)(4-\partial_x^2)\chi $. Moreover, by the Plancheral's theorem, we have
\begin{equation*}
\begin{split}
(m_t,\chi)_{L^2}=(\mathcal {F}m_t,\mathcal {F}\chi)_{L^2}=&((1+\eta^2)(4+\eta^2) \mathcal {F}\chi_t,\mathcal {F}\chi)_{L^2}\\
=&( \mathcal {F}\chi_t,(1+\eta^2)(4+\eta^2)\mathcal {F}\chi)_{L^2}=(\chi_t,m)_{L^2}.
\end{split}
\end{equation*}
In terms of the first component of (4.1), we have
\begin{equation}
\begin{split}
\frac{d}{dt}(m,\chi)_{L^2}&=(m_t,\chi)_{L^2}+(m,\chi_t)_{L^2}=2(m_t,\chi)_{L^2} \\
&=-(2\alpha(t)(um)_x,\chi)_{L^2}-(4\alpha(t) u_xm,\chi)_{L^2}+  (\gamma(t)\sigma^2,\chi_x)_{L^2}.
\end{split}
\end{equation}
Note that
\begin{equation*}
\begin{split}
-(2\alpha(t)(um)_x,\chi)_{L^2}&=(2\alpha(t)u^2,\chi_x)_{L^2}+(2\alpha(t)u_{xx},u\chi_x)_{L^2}\\
&=(2\alpha(t)u^2,\chi_x)_{L^2}+(2\alpha(t)u_{x}^2,\chi_x)_{L^2}-( \alpha(t)u^2, \chi_{xxx})_{L^2}\\
&=(2\alpha(t)u^2,\chi_x)_{L^2}+(2\alpha(t)u_{x}^2,\chi_x)_{L^2}-(\alpha(t)u^2, 4\chi_x-u_x)_{L^2}\\
&=-(2\alpha(t)u^2,\chi_x)_{L^2}+(2\alpha(t)u_{x}^2,\chi_x)_{L^2},
\end{split}
\end{equation*}
and
\begin{equation*}
\begin{split}
-(4\alpha(t) u_xm,\chi)_{L^2}=(2\alpha(t) u^2  ,\chi_x)_{L^2}-(2\alpha(t) u_x^2 ,\chi_x)_{L^2}.
\end{split}
\end{equation*}
Therefore we get from (4.5) that
\begin{equation}
\begin{split}
 (m(t),\chi(t))_{L^2} = (m_0,\chi_0)_{L^2}+ \int_0^t(\gamma(t')\sigma^2(t'),\chi_x(t'))_{L^2}dt'.
\end{split}
\end{equation}
Since the Fourier transformation is an isometric map from $L^2$ into $L^2$, we have
\begin{equation*}
\begin{split}
 \| u(t)\|_{L^2}^2\geq(m(t),\chi(t))_{L^2}&=(\mathcal {F}m(t),\mathcal {F}\chi(t))_{L^2}\\
 &=((1+\eta^2)\mathcal {F}u(t,\eta),(4+\eta^2)^{-1}\mathcal {F}u(t,\eta))_{L^2}\\
 &\geq \frac{1}{4}\|\mathcal {F}u(t,\eta)\|_{L^2}^2=\frac{1}{4}\| u(t)\|_{L^2}^2,
\end{split}
\end{equation*}
and it follows from the H\"{o}lder inequality that
\begin{equation*}
\begin{split}
\int_0^t(\gamma(t')\sigma^2(t'),\chi_x(t'))_{L^2}dt'&\leq \int_0^t|\gamma(t')|\|\sigma^2(t')\|_{L^2}\|\chi_x(t')\|_{L^2}dt'\\
&=\int_0^t|\gamma(t')|\|\sigma(t')\|_{L^4}^2\|\eta(4+\eta^2)^{-1}\mathcal {F}u(t',\eta)\|_{L^2}dt'\\
&\leq\int_0^t|\gamma(t')|\|\sigma(t')\|_{L^2}^2\|\sigma(t')\|_{L^\infty}^2\|u(t')\|_{L^2}dt'\\
&\leq\|\sigma_0 \|_{H^{s-1}} ^4\int_0^t|\gamma(t')|e^{-3\int_0^t\inf_{x\in \mathbb{K}} \{\xi(t') u_x(t',x)\}dt'}\|u(t')\|_{L^2}dt'.
\end{split}
\end{equation*}
Inserting above two estimates into (4.6), we get
\begin{equation*}
\begin{split}
\| u(t)\|_{L^2}^2\leq 4\| u_0\|_{L^2}^2+ 4\|\sigma_0 \|_{H^{s-1}} ^4\int_0^t|\gamma(t')|e^{-3\int_0^t\inf_{x\in \mathbb{K}} \{\xi(t') u_x(t',x)\}dt'}\|u(t')\|_{L^2}dt',
\end{split}
\end{equation*}
which combined with the Gronwall's lemma yields the $L^2$-bound for $u$.

To derive the $L^\infty$-bound for $u$, we utilize the equivalent form of the first component:
\begin{equation*}
 u_t+\alpha(t) uu_x=-\partial_xp*\left(\frac{3\alpha(t)}{2} u^2+\frac{\gamma(t)}{2} \sigma^2 \right) .
\end{equation*}
Along the characteristic $\psi(t,x)$, we have
\begin{equation*}
\begin{split}
\left \|\frac{d u(t,\psi(t,x))}{dt}\right\|_{L^\infty}&=\left\|\partial_xp*\left(\frac{3\alpha(t)}{2} u^2+\frac{\gamma(t)}{2} \sigma^2 \right)(t,\psi(t,x))\right\|_{L^\infty}\\
&\leq \left\|\partial_xp\right\|_{L^\infty}\left\| \frac{3\alpha(t)}{2} u^2+\frac{\gamma(t)}{2} \sigma^2\right\|_{L^1}\\
&\leq  |\alpha(t)|\| u \|_{L^2}^2+|\gamma(t)| \|\sigma \|_{L^2}^2\\
&\leq4\| u_0\|_{L^2}^2|\alpha(t)|\exp\left\{4\|\sigma_0 \|_{H^{s-1}} ^4 \|\gamma\|_{L^1(0,t)}e^{-3\int_0^t\inf_{x\in \mathbb{K}} \{\xi(t') u_x(t',x)\}dt'}\right\}\\
&\quad +\|\sigma_0\|_{L^2}^2|\gamma(t)| e^{- \int_0^t\inf_{x\in \mathbb{K}} \{\xi(t')u_x(t',x)\}  dt'}\\
&= \mathcal {J}(t).
\end{split}
\end{equation*}
Integrating with respect to $t$, we get from the above estimate that
\begin{equation*}
\begin{split}
 -\int_0^t\mathcal {J}(t')dt'\leq u(t,\psi(t,x)) -u_0\leq \int_0^t\mathcal {J}(t')dt'.
\end{split}
\end{equation*}
By using the homeomorphism property of $\psi(t,x)$ in $x$-variable, we obtain the $L^\infty$-bound for $u$. The proof of Lemma is completed.
\end{proof}

Based on th above properties for $u$ and $\sigma$, we can now give the proof of Theorem \ref{theorem4}.
\begin{proof} [Proof of Theorem \ref{theorem3}] We prove the theorem by an inductive argument with respect to the space regularity index $s$.
\vskip3mm\noindent
\textbf{Step 1 ($\frac{3}{2}< s < 2$):}  By applying Lemma 2.5 to the transport equation with respect to $\sigma$:
\begin{equation}
\begin{split}
 \sigma_t+\xi(t)u\sigma_x+\xi(t)\sigma u_x=0,
\end{split}
\end{equation}
we have  (for every $1 <s< 2$, indeed)
\begin{equation*}
\begin{split}
  \|\sigma(t)\|_{H^{s-1}}\leq  \|\sigma_0\|_{H^{s-1}}+C\int_0^t \|\xi(t')\sigma u_x\|_{H^{s-1}}dt'+C\int_0^t|\xi(t')| \|\sigma\|_{H^{s-1}}(\|u\|_{L^\infty}+\|u_x\|_{L^\infty})dt'.
\end{split}
\end{equation*}
Thanks to the Moser estimate (cf. Corollary 2.8 in \cite{64}), we have
\begin{equation*}
\begin{split}
 \|\xi(t)\sigma u_x\|_{H^{s-1}}\leq |\xi(t)|(\|\sigma \|_{L^\infty}\|u_x\|_{H^{s-1}}+\|u_x\|_{L^\infty}\|\sigma\|_{H^{s-1}}).
\end{split}
\end{equation*}
 Therefore, we have
\begin{equation}
\begin{split}
  \|\sigma(t)\|_{H^{s-1}}\leq&  \|\sigma_0\|_{H^{s-1}}+C\int_0^t |\xi(t')|(\|\sigma \|_{L^\infty}\|u\|_{H^{s}}+\|u_x\|_{L^\infty}\|\sigma\|_{H^{s-1}})dt'\\
  &+C\int_0^t |\xi(t')|(\|u\|_{L^\infty}+\|u_x\|_{L^\infty})\|\sigma\|_{H^{s-1}}dt'.
\end{split}
\end{equation}
On the other hand, applying Lemma  2.3 to the equation
\begin{equation}
\begin{split}
  u_t+\alpha(t) uu_x=-\partial_xp*\left(\frac{3\alpha (t)}{2} u^2+\frac{\gamma(t)}{2} \sigma^2\right)
\end{split}
\end{equation}
implies (for every $s > 1$, indeed)
\begin{equation*}
\begin{split}
  \|u(t)\|_{H^{s}}\leq&  \|u_0\|_{H^{s}}+C\int_0^t\left \|\partial_xp*\left(\frac{3\alpha }{2} u^2+\frac{\gamma}{2} \sigma^2\right)(t')\right\|_{H^{s}}dt'\\
  &+C\int_0^t|\alpha(t')|\|u_x(t')\|_{L^\infty}\|u(t')\|_{H^{s}}dt'.
\end{split}
\end{equation*}
Thanks to the Moser estimate for $s-1>0$, one has
\begin{equation*}
\begin{split}
\left \|\partial_xp*\left(\frac{3\alpha }{2} u^2+\frac{\gamma}{2} \sigma^2\right) \right\|_{H^{s}}
\leq  C(|\alpha(t)|\| u\|_{L^\infty}\| u\|_{H^{s-1}}+ |\gamma(t)|\| \sigma\|_{L^\infty}\| \sigma\|_{H^{s-1}}).
\end{split}
\end{equation*}
From this, we reach
\begin{equation}
\begin{split}
  \|u(t)\|_{H^{s}}\leq&  \|u_0\|_{H^{s}}+C\int_0^t(|\alpha(t')|\| u(t')\|_{L^\infty}\| u(t')\|_{H^{s-1}}+ |\gamma(t')|\| \sigma(t')\|_{L^\infty}\| \sigma(t')\|_{H^{s-1}})dt'\\
  &+C\int_0^t|\alpha(t')| \|u_x(t')\|_{L^\infty}\|u(t')\|_{H^{s}}dt'.
\end{split}
\end{equation}
which together with (4.8) yields that
\begin{equation}
\begin{split}
&\|u(t)\|_{H^{s}}+\|\sigma(t)\|_{H^{s-1}}\leq  \|u_0\|_{H^{s}}+\|\sigma_0\|_{H^{s-1}} \\
  & \quad  +C\int_0^t (|\alpha(t')|+|\xi(t')|+|\gamma(t')|)(\|u\|_{L^\infty}+\|u_x\|_{L^\infty}+\| \sigma\|_{L^\infty}) (\| u\|_{H^{s}}+ \|\sigma\|_{H^{s-1}}) dt' .
\end{split}
\end{equation}
An application of the  Gronwall's lemma leads to
\begin{equation}
\begin{split}
\|u(t)\|_{H^{s}}+\|\sigma(t)\|_{H^{s-1}}\leq (\|u_0\|_{H^{s}}+\|\sigma_0\|_{H^{s-1}})e^{C\int_0^t (|\alpha(t')|+|\xi(t')|+|\gamma(t')|)  (\|u\|_{L^\infty}+\|u_x\|_{L^\infty}+\| \sigma\|_{L^\infty}) dt'}.
\end{split}
\end{equation}
In terms of the estimate (4.4) and Lemma 4.1, we see that
$$
 \|\sigma(t)\|_{L^\infty}\leq \|\sigma_0 \|_{H^{s-1}} e^{ \int_0^t |\xi(t')| \|u_x(t')\|_{L^\infty} dt'},
$$
and
\begin{equation*}
\begin{split}
  \|u(t)\|_{L^\infty} \leq& \|u_0\|_{L^\infty}+  4\| u_0\|_{L^2}^2t|\alpha(t)|\exp\left\{4\|\sigma_0 \|_{H^{s-1}} ^4\|\gamma\|_{L^1(0,t)}e^{3\int_0^t|\xi(t')|\|u_x(t')\|_{L^\infty}dt'}\right\}\\
& +\|\sigma_0\|_{L^2}^2t|\gamma(t)| e^{\int_0^t|\xi(t')|\|u_x(t')\|_{L^\infty}dt'} .
\end{split}
\end{equation*}
If $\int_0^{T^*} (|\alpha(t')|+|\xi(t')|+|\gamma(t')|) \|u_x(t')\|_{L^\infty} dt'<\infty$, then it  follows from (4.12) and the fact of $\alpha,\gamma,\xi \in L^2_{loc}([0,\infty);\mathbb{R})$ that
\begin{equation*}
\begin{split}
\lim\sup_{t\rightarrow T^*}(\|u(t)\|_{H^{s}}+\|\sigma(t)\|_{H^{s-1}})< \infty,
\end{split}
\end{equation*}
which is incompatible with the assumption that $T^*$ is the maximum existence time of solutions.

\vskip3mm\noindent
\textbf{Step 2 ($2\leq s < \frac{5}{2}$):}  Applying Lemma \ref{lem:priorieti} to Eq.(4.7), we get
\begin{equation}
\begin{split}
  \|\sigma(t)\|_{H^{s-1}}\leq&  \|\sigma_0\|_{H^{s-1}}+C\int_0^t |\xi(t')| \|\sigma \|_{L^\infty}\|u\|_{H^{s}} dt' +C\int_0^t |\xi(t')| \|u_x\|_{L^\infty\cap H^\frac{1}{2}} \|\sigma\|_{H^{s-1}}dt'.
\end{split}
\end{equation}
which together with (4.10) yields
\begin{equation*}
\begin{split}
  & \|u(t)\|_{H^{s}}+ \|\sigma(t)\|_{H^{s-1}}\leq \|u_0\|_{H^{s}}+\|\sigma_0\|_{H^{s-1}} \\
  &\quad +C\int_0^t(|\alpha(t')|+|\xi(t')|+|\gamma(t')|)  (\|u \|_{H^{\frac{3}{2}+\epsilon}}+\|\sigma \|_{L^\infty}) (\|u\|_{H^{s}}+\|\sigma\|_{H^{s-1}})dt',
\end{split}
\end{equation*}
for any $\epsilon\in (0,\frac{1}{2})$, where we used the fact of $H^{\frac{3}{2}+\epsilon}\hookrightarrow L^\infty\cap H^\frac{1}{2}$.

Applying the Gronwall's lemma to above inequality leads to
\begin{equation}
\begin{split}
 \|u(t)\|_{H^{s}}+ \|\sigma(t)\|_{H^{s-1}}\leq (\|u_0\|_{H^{s}}+\|\sigma_0\|_{H^{s-1}})
 e^{C\int_0^t(|\alpha(t')|+|\xi(t')|+|\gamma(t')|)  (\|u \|_{H^{\frac{3}{2}+\epsilon}}+\|\sigma \|_{L^\infty}) dt'}.
\end{split}
\end{equation}
Therefore, thanks to the uniqueness of solution in Theorem \ref{theorem1}, we get that: if the maximal existence time $T^*$ satisfies $\int_0^{T^*}(|\alpha(t')|+|\xi(t')|+|\gamma(t')|) \|u_x (t')\|_{L^\infty} dt'< \infty$, then Step 1 with $\frac{3}{2}< \frac{3}{2}+\epsilon < 2$ and the fact of $
 \|\sigma(t)\|_{L^\infty}\leq \|\sigma_0 \|_{H^{s-1}} e^{ \int_0^t |\xi(t')| \|u_x(t')\|_{L^\infty} dt'},
$ imply that
\begin{equation*}
\begin{split}
\lim\sup_{t\rightarrow T^*}(\|u(t)\|_{H^{s}}+\|\sigma(t)\|_{H^{s-1}})< \infty
\end{split}
\end{equation*}
contradicts the assumption that $T^*$ is the maximum existence time of solutions.

\vskip3mm\noindent
\textbf{Step 3 ($2< s < 3$):}
Differentiating the Eq.(4.7) with respect to $x$, we have
\begin{equation*}
\begin{split}
 \partial_t\sigma_x+\xi(t)u\partial_x\sigma_x+2\xi(t)u_x \sigma_x+\xi(t)\sigma u_{xx}=0.
\end{split}
\end{equation*}
Applying   Lemma 2.5 to the above equation implies that
\begin{eqnarray}
\begin{split}
  \|\sigma_x(t)\|_{H^{s-2}}\leq&  \|\partial_x\sigma_0\|_{H^{s-2}}+C\int_0^t |\xi(t')| \|(2 u_x \sigma_x+ \sigma u_{xx})(t')\|_{H^{s-2}} dt' \\
  &+C\int_0^t |\xi(t')|(\|u\|_{L^\infty}+\|u_x\|_{L^\infty})\|\sigma_x\|_{H^{s-2}}dt'\\
  \leq & C\int_0^t |\xi(t')|(\|u\|_{L^\infty}+\|u_x\|_{L^\infty}+\|\sigma\|_{L^\infty})(\|u(t')\|_{H^{s}}+\|\sigma(t')\|_{H^{s-1}})dt',\\
\end{split}
\end{eqnarray}
where we used the following Moser-type estimates:
\begin{equation*}
\begin{split}
 \| u_x \sigma_x \|_{H^{s-2}}  \leq  C\left(\|\sigma\|_{L^\infty}\| u_x \|_{H^{s-1}}+ \|\sigma_x\|_{H^{s-2}}\| u_x\|_{L^\infty}    \right),
\end{split}
\end{equation*}
and
\begin{equation*}
\begin{split}
 \|\sigma u_{xx} \|_{H^{s-2}}  \leq  C\left(\|\sigma\|_{H^{s-1}}\| u_{x} \|_{L^\infty}+ \|u_{xx}\|_{H^{s-2}}\|\sigma\|_{L^\infty}    \right).
\end{split}
\end{equation*}
(4.15), together with (4.18) and (4.16) (where $s-1$ is replaced by $s-2$), implies that
\begin{equation*}
\begin{split}
&\|u(t)\|_{H^{s}}+\|\sigma(t)\|_{H^{s-1}}\leq   \|u_0\|_{H^{s}}+\|\sigma_0\|_{H^{s-1}} \\
  & \quad  +C\int_0^t (|\alpha(t')|+|\xi(t')|+|\gamma(t')|) (\|u\|_{L^\infty}+\|u_x\|_{L^\infty}+\| \sigma\|_{L^\infty}) (\| u\|_{H^{s}}+ \|\sigma\|_{H^{s-1}}) dt' .
\end{split}
\end{equation*}
Applying the Gronwall's inequality again gives (4.10). Hence, using the arguments as in Step 1, it completes the proof of Theorem 4.1 for $2 <s< 3$.

\vskip3mm\noindent
\textbf{Step 4 ($s =k \in \mathbb{N}$, $k \geq3$):} Differentiating the Eq.(4.15) $k-2$ times with respect to $x$, we get
\begin{equation}
\begin{split}
\partial _t\partial_x^{k-2}\sigma+\xi(t) u\partial_x\partial_x^{k-2}\sigma+\xi(t)\sum_{l=1} ^{k-2}C_{k-1}^l\partial_x^{l}u \partial_x^{k-l-1}\sigma + \sigma\partial_x^{k-1}u=0,
\end{split}
\end{equation}
Applying Theorem 3.1 to the transport equation (4.24), we have
\begin{equation}
\begin{split}
  \|\partial_x^{k-2}\sigma(t)\|_{H^1}\leq&  \|\partial_x^{k-2} \sigma_0\|_{H^1}+C\int_0^t\|\partial_x^{k-2}\sigma(t')\|_{H^1} \|\xi(t) \partial_x u(t')\|_{L^\infty\cap H^{\frac{1}{2}}} dt'\\
  & +C\int_0^t\left\|\xi(t)\sum_{l=1} ^{k-2}C_{k-1}^l\partial_x^{l}u \partial_x^{k-l-1}\sigma+ \sigma\partial_x^{k-1}u\right\|_{ H^{1}} dt'\\
\leq &\|\sigma_0\|_{H^{k-1}} +C\int_0^t|\xi(t')| (\| u (t')\|_{ H^{s}}+\| \sigma (t')\|_{ H^{s-1}})(\| u (t')\|_{ H^{s-1}}+\| \sigma (t')\|_{ H^{1}}) dt',
\end{split}
\end{equation}
where we have used the following facts (as $H^1$ is an algebra):
\begin{equation*}
\begin{split}
 &\left\|\xi(t)\sum_{l=1} ^{k-2}C_{k-1}^l\partial_x^{l}u \partial_x^{k-l-1}\sigma \right\|_{ H^{1}}\\
 &\quad \leq  |\xi(t)|\sum_{l=1} ^{k-2}C_{k-1}^l\|\partial_x^{l}u\|_{ H^{1}}\| \partial_x^{k-l-1}\sigma  \|_{ H^{1}}\leq C |\xi(t)| \| u \|_{ H^{s-1}}\| \sigma \|_{ H^{s-1}},
\end{split}
\end{equation*}
and
\begin{equation*}
\begin{split}
\|\sigma\partial_x^{k-1}u\|_{ H^{1}} \leq \|\sigma\|_{ H^{1}}\|\partial_x^{k-1}u\|_{ H^{1}}\leq C\|\sigma\|_{ H^{1}}\|u\|_{ H^{s}}.
\end{split}
\end{equation*}
(4.18), together with (4.10) and (4.8) (where $s-1$ is replaced by 1), implies that
\begin{equation*}
\begin{split}
   &\| u(t)\|_{H^{s}}+\| \sigma(t)\|_{H^{s-1}} \leq  \|u_0\|_{H^{s}} +\|\sigma_0\|_{H^{s-1}}\\
   &\quad +C\int_0^t(|\alpha(t')|+|\xi(t')|+|\gamma(t')|)(\| u (t')\|_{ H^{s-1}}+\| \sigma (t')\|_{ H^{1}}) (\| u (t')\|_{ H^{s}}+\| \sigma (t')\|_{ H^{s-1}})dt'.
\end{split}
\end{equation*}
Applying the Gronwall's lemma to the above inequality leads to
\begin{equation}
\begin{split}
\| u(t)\|_{H^{s}}+\| \sigma(t)\|_{H^{s-1}}\leq (\|u_0\|_{H^{s}} +\|\sigma_0\|_{H^{s-1}})e^{C\int_0^t|\xi(t')| (\| u (t')\|_{ H^{s-1}}+\| \sigma (t')\|_{ H^{1}}) dt'}
\end{split}
\end{equation}
If the maximal existence time $T^*$ satisfies $\int_0^{T^*}(|\alpha(t')|+|\xi(t')|+|\gamma(t')|) \|u_x (t')\|_{L^\infty} dt'< C$ for some  positive constant $C$,  thanks to
the uniqueness of solution in Theorem \ref{theorem1}, we get that $ \| u (t)\|_{ H^{s-1}}+\| \sigma (t)\|_{ H^{1}}$ is uniformly bounded by the induction assumption, which together with (4.18) and the fact of $ L^2_{loc}([0,\infty);\mathbb{R})\subset L^1_{loc}([0,\infty);\mathbb{R})$ implies
\begin{equation*}
\begin{split}
\lim\sup_{t\rightarrow T^*}(\|u(t)\|_{H^{s}}+\|\sigma(t)\|_{H^{s-1}})< \infty.
\end{split}
\end{equation*}
This leads to a contradiction.

\vskip3mm\noindent
\textbf{Step 5 ($k<s<k+1$, $k\in\mathbb{N}$, $k\geq3$):} Differentiating the Eq. (4.15) $k-1$ times with respect to $x$, we get
\begin{equation}
\begin{split}
\partial _t\partial_x^{k-1}\sigma+\xi(t) u\partial_x\partial_x^{k-1}\sigma+\xi(t)\sum_{l=1} ^{k-1}C_{k}^l\partial_x^{l}u \partial_x^{k-l}\sigma + \xi(t)\sigma\partial_x\partial_x^{k-1}u=0.
\end{split}
\end{equation}
Applying Lemma 2.5 to (4.19) implies that
\begin{equation*}
\begin{split}
  \|\partial_x^{k-1}\sigma(t)\|_{H^{s-k}}\leq&  \|\partial_x^{k-1}\sigma_0\|_{H^{s-k}}+C\int_0^t \left\|\xi(t)\sum_{l=1} ^{k-1}C_{k}^l\partial_x^{l}u \partial_x^{k-l}\sigma +\xi(t) \sigma\partial_x\partial_x^{k-1}u\right\|_{H^{s-k}}dt'\\
  &+C\int_0^t|\xi(t')| \|\partial_x^{k-1}\sigma\|_{H^{s-k}}(\|u\|_{L^\infty}+\|u_x\|_{L^\infty})dt'.
\end{split}
\end{equation*}
Using the Moser-type estimate and the Sobolev embedding theorem, we have for $0 <\epsilon < \frac{1}{2}$
\begin{equation*}
\begin{split}
& \left\|\xi(t)\sum_{l=1} ^{k-1}C_{k}^l\partial_x^{l}u \partial_x^{k-l}\sigma \right\|_{H^{s-k}}\\
  &\quad \leq C|\xi(t)| \sum_{l=1} ^{k-1}(\|\partial_x^{l}u\|_{H^{s-k+1}}\|\partial_x^{k-l-1}\sigma\|_{L^\infty}+\|\partial_x^{k-l}\sigma\|_{H^{s-k}}\|\partial_x^{l}u\|_{L^\infty})\\
  &\quad \leq C|\xi(t)| \sum_{l=1} ^{k-1}(\|\partial_x^{l}u\|_{H^{s-k+1}}\|\partial_x^{k-l-1}\sigma\|_{H^{\frac{1}{2}+\epsilon}}+\|\partial_x^{k-l}\sigma\|_{H^{s-k}}\|\partial_x^{l}u\|_{H^{\frac{1}{2}+\epsilon}})\\
  &\quad \leq C|\xi(t)| (\| u\|_{H^{s}}\|\sigma\|_{H^{k-\frac{3}{2}+\epsilon}}+\| \sigma\|_{H^{s-1}}\| u\|_{H^{k-\frac{1}{2}+\epsilon}})
\end{split}
\end{equation*}
and
\begin{equation*}
\begin{split}
 \| \xi(t) \sigma\partial_x\partial_x^{k-1}u \|_{H^{s-k}}&\leq   C|\xi(t)|  (\|\partial_x^{k-1}u\|_{H^{s-k}}\| \sigma\|_{L^\infty}+\| \sigma\|_{H^{s-k}}\|\partial_x^{k-1}u\|_{L^\infty})\\
 &\leq   C|\xi(t)|  (\| u\|_{H^{s-1}}\| \sigma\|_{L^\infty}+\| \sigma\|_{H^{s-k}}\| u\|_{H^{k-\frac{1}{2}+\epsilon}}).
\end{split}
\end{equation*}
Hence, we get
\begin{equation*}
\begin{split}
  \|\partial_x^{k-1}\sigma(t)\|_{H^{s-k}}\leq&  \| \sigma_0\|_{H^{s-1}}+C\int_0^t |\xi(t')| (\| u\|_{H^{s}} +\| \sigma\|_{H^{s-1}} )  (\| u\|_{H^{k-\frac{1}{2}+\epsilon}}+\|\sigma\|_{H^{k-\frac{3}{2}+\epsilon}})dt',
\end{split}
\end{equation*}
which together with (4.10) and (4.8) (where $ s-1$ is replaced by $s-k$) implies that
\begin{equation*}
\begin{split}
 &\|u(t)\|_{H^{s}}+  \| \sigma(t)\|_{H^{s-k}}\leq  \|u_0\|_{H^{s}}+ \| \sigma_0\|_{H^{s-1}}\\
 &\quad +C\int_0^t(|\alpha(t')|+|\xi(t')|+|\gamma(t')|) (\| u\|_{H^{s}} +\| \sigma\|_{H^{s-1}} )  (\| u\|_{H^{k-\frac{1}{2}+\epsilon}}+\|\sigma\|_{H^{k-\frac{3}{2}+\epsilon}})dt',
\end{split}
\end{equation*}
Applying Gronwall's inequality then gives that
\begin{equation}
\begin{split}
\|u(t)\|_{H^{s}}+  \| \sigma(t)\|_{H^{s-k}}\leq  (\|u_0\|_{H^{s}}+ \| \sigma_0\|_{H^{s-1}})e^{C\int_0^t(|\alpha(t')|+|\xi(t')|+|\gamma(t')|) (\| u\|_{H^{k-\frac{1}{2}+\epsilon}}+\|\sigma\|_{H^{k-\frac{3}{2}+\epsilon}})dt'}.
\end{split}
\end{equation}
If $\int_0^{T^*}(|\alpha(t')|+|\xi(t')|+|\gamma(t')|) \|u_x (t')\|_{L^\infty} dt'< \infty$,  thanks to
the uniqueness of solution in Theorem \ref{theorem1}, we get that $\| u\|_{H^{k-\frac{1}{2}+\epsilon}}+\|\sigma\|_{H^{k-\frac{3}{2}+\epsilon}}$ is uniformly bounded by the induction assumption, which together with (4.20) and the fact of $ L^2_{loc}([0,\infty);\mathbb{R})\subset L^1_{loc}([0,\infty);\mathbb{R})$ implies
\begin{equation*}
\begin{split}
\lim\sup_{t\rightarrow T^*}(\|u(t)\|_{H^{s}}+\|\sigma(t)\|_{H^{s-1}})< \infty.
\end{split}
\end{equation*}
which leads to a contradiction. Therefore, from Step 1 to Step 5, we complete the proof of the blow-up criteria.

 To derive a lower bound for the blow-up time, it follows from the above analysis and the Sobolev embedding $H^s\hookrightarrow L^\infty$ for $s>\frac{1}{2}$ that
\begin{equation*}
\begin{split}
 &\|u(t)\|_{H^{s}}+  \| \sigma(t)\|_{H^{s-k}}\leq  \|u_0\|_{H^{s}}+ \| \sigma_0\|_{H^{s-1}}\\
 &\quad +C\int_0^t(|\alpha(t')|+|\xi(t')|+|\gamma(t')|) (\| u(t')\|_{H^{s}} +\| \sigma(t')\|_{H^{s-1}} ) ^2dt'.
\end{split}
\end{equation*}
We find by solving the above inequality that
\begin{equation}
\begin{split}
 \|u(t)\|_{H^{s}}+  \| \sigma(t)\|_{H^{s-k}}\leq   \left(1-C(\|u_0\|_{H^{s}}+ \| \sigma_0\|_{H^{s-1}})\int_0^t  (|\alpha(t')|+|\xi(t')|+|\gamma(t')|) dt'\right)^{-1}.
\end{split}
\end{equation}
From the right hand side of (4.21), we have the lower bound of blow-up time $T^*$
\begin{equation*}
\begin{split}
T^*\geq T(u_0,\sigma_0)\overset{\text{def}}{=}\sup_{t>0}\left\{t>0;~~\int_0^{t}(|\alpha(t')|+|\xi(t')|+|\gamma(t')|) dt'\leq\frac{1}{C(\|u_0\|_{H^{s}}+ \| \sigma_0\|_{H^{s-1}})}\right\}.
\end{split}
\end{equation*}
Therefore, if $\alpha,\beta,\gamma\in L^1([0,\infty);\mathbb{R})$ satisfies
$$
 \int_0^t  (|\alpha(t')|+|\xi(t')|+|\gamma(t')|) dt'< \frac{1}{C(\|u_0\|_{H^{s}}+ \| \sigma_0\|_{H^{s-1}})},
$$
we get
$$
T^*=T(u_0,\sigma_0)=\infty,
$$
which together with (4.21) implies that
\begin{equation*}
\begin{split}
\sup_{t\in [0,\infty)}(\|u(t)\|_{H^{s}}+\|\sigma(t)\|_{H^{s-1}})< \infty.
\end{split}
\end{equation*}
In other words, the solution $(u,\sigma)$ exists globally. The proof of Theorem is now completed.
\end{proof}

\begin{proof} [Proof of Theorem \ref{theorem4}]
Assume that the solution $(u,\sigma)$ blows up in finite time $(T^* <\infty)$ and there exists an $M > 0$ such that $\inf_{x\in \mathbb{K}} u_x(t,x)\geq -M, \quad  \forall t \in [0,T^*)$. Since $\xi(t)\geq 0$ and  $\alpha(t)+\gamma(t)+\xi(t)\geq 0$, we have
\begin{equation}
\begin{split}
-\inf_{x\in \mathbb{K}} \{\xi(t)u_x(t,x)\}\leq M \xi(t) , \quad  \forall t \in [0,T^*),
\end{split}
\end{equation}
and
\begin{equation}
\begin{split}
-\inf_{x\in \mathbb{K}} \{(\alpha(t)+\gamma(t)+\xi(t))u_x(t,x)\}\leq M(\alpha(t)+\gamma(t)+\xi(t)), \quad  \forall t \in [0,T^*).
\end{split}
\end{equation}
It then follows from (4.22) and Lemma 4.1 that
\begin{equation*}
\begin{split}
 \|\sigma(t,\cdot)\|_{L^\infty} \leq \|\sigma_0 \|_{H^{s-1}} e^{M\|\xi\|_{L^1(0,t)}},
\end{split}
\end{equation*}
and
\begin{equation*}
\begin{split}
  \|u(t)\|_{L^\infty} \leq& \|u_0\|_{L^\infty}+  4\| u_0\|_{L^2}^2t|\alpha(t)|\exp\left\{4\|\sigma_0 \|_{H^{s-1}} ^4\|\gamma\|_{L^1(0,t)}e^{ 3M\|\xi\|_{L^1(0,t)}}\right\}\\
& +\|\sigma_0\|_{L^2}^2t|\gamma(t)| e^{M\|\xi\|_{L^1(0,t)}}.
\end{split}
\end{equation*}
Differentiating the first component of system \eqref{1.6}, we have
\begin{equation}
\begin{split}
 \partial_t u_x+\alpha(t) u \partial_xu_x=&- p*\left(\frac{3\alpha(t)}{2} u^2+\frac{\gamma(t)}{2} \sigma^2\right) +\frac{3\alpha(t)}{2} u^2+\frac{\gamma(t)}{2} \sigma^2-\alpha(t)u_x^2 .
\end{split}
\end{equation}
Along the characteristic curve $\psi(t,x)$, it is seen that
\begin{equation*}
\begin{split}
  \frac{du_x(t,\psi(t,x))}{dt}=&- p*\left(\frac{3\alpha}{2} u^2+\frac{\gamma}{2} \sigma^2 \right)(t,\psi(t,x)) +\left(\frac{3\alpha}{2} u^2+\frac{\gamma}{2} \sigma^2-\alpha u_x^2\right)(t,\psi(t,x))\\
  &\leq \left(\frac{3\alpha}{2} u^2+\frac{\gamma}{2} \sigma^2 \right)(t,\psi(t,x))\\
  &\leq \frac{3|\beta(t)| }{2}\|u(t)\|_{L^\infty}+\frac{ |\gamma(t)| }{2}\|\sigma(t)\|_{L^\infty}\\
  &\leq  \frac{3 \|u_0\|_{H^{s}}}{2}|\beta(t)|+  6\| u_0\|_{L^2}^2 t|\beta(t)||\alpha(t)|\exp\left\{4\|\sigma_0 \|_{H^{s-1}} ^4\|\gamma\|_{L^1(0,t)}e^{ 3M\|\xi\|_{L^1(0,t)}}\right\}\\
& \quad +\frac{3 \|\sigma_0\|_{H^{s-1}}^2}{2}t|\beta(t)||\gamma(t)| e^{M\|\xi\|_{L^1(0,t)}} +\frac{ \|\sigma_0 \|_{H^{s-1}} }{2} |\gamma(t)|e^{M\|\xi\|_{L^1(0,t)}}.
\end{split}
\end{equation*}
Integrating the above inequality on $[0,t]$, we obtain
\begin{equation}
\begin{split}
   u_x(t,\psi(t,x))-u_{0,x}  &\leq  \frac{3 \|u_0\|_{H^{s}}}{2}\|\beta\|_{L^1(0,t)}+\frac{3 \|\sigma_0\|_{H^{s-1}}^2}{2} t e^{M\|\xi\|_{L^1(0,t)}}(\|\beta\|_{L^1(0,t)}^2+\|\gamma\|_{L^1(0,t)}^2)\\
   &\quad +  6\| u_0\|_{L^2}^2 t \exp\left\{4\|\sigma_0 \|_{H^{s-1}} ^4\|\gamma\|_{L^1(0,t)}e^{ 3M\|\xi\|_{L^1(0,t)}}\right\}(\|\beta\|_{L^1(0,t)}^2+\|\gamma\|_{L^1(0,t)}^2)\\
& \quad +\frac{ \|\sigma_0 \|_{H^{s-1}} }{2}\|\gamma\|_{L^1(0,t)}e^{M\|\xi\|_{L^1(0,t)}}\\
& \overset{\text{def}}{=} \mathcal {L}(t)\leq \mathcal {L}(T^*)<\infty.
\end{split}
\end{equation}
Note that $\mathcal {L}(t)$ is an increasing and absolutely continuous function. Since $\alpha,\gamma,\xi \in L^2_{loc}([0,\infty);\mathbb{R})\subset L^1_{loc}([0,\infty);\mathbb{R})$, it follows from (4.25) that
\begin{equation*}
\begin{split}
  \sup_{x\in \mathbb{K}} u_x(t,x) =\sup_{x\in \mathbb{K}} u_x(t,\psi(t,x)) \leq \|u_{0,x} \|_{L^\infty}+\mathcal {L}(T^*), \quad \forall t\in [0,T^*),
\end{split}
\end{equation*}
which together with the fact of $\alpha(t)+\gamma(t)+\xi(t)\geq 0$ implies that
\begin{equation}
\begin{split}
  \sup_{x\in \mathbb{K}} \{(\alpha(t)+\gamma(t)+\xi(t))u_x(t,x)\}  \leq(\|u_{0} \|_{H^{s}}+\mathcal {L}(T^*))(\alpha(t)+\gamma(t)+\xi(t)),
\end{split}
\end{equation}
for any $t\in [0,T^*)$. As a result, we get from (4.23) and (4.26) that
\begin{equation*}
\begin{split}
&\int_0^{T^*}(\alpha(t')+\gamma(t')+\xi(t'))\|u_x(t')\|_{L^\infty}dt'\\
&\quad \leq (M+\|u_{0} \|_{H^{s}}+\mathcal {L}(T^*))\int_0^{T^*}(\alpha(t')+\gamma(t')+\xi(t'))dt'<\infty,
\end{split}
\end{equation*}
Theorem \ref{theorem4} implies that the maximal existence time $T^*=\infty$, which contradicts the assumption on the maximal existence time $T^*<\infty$.

Conversely, the Sobolev embedding $H^s\hookrightarrow L^\infty$ with $s > \frac{1}{2}$  implies that if $\lim_{t\rightarrow T^*}\inf_{x\in \mathbb{K}} u_x(t,x)=-\infty$ holds, then the corresponding solution blows up in finite time, which completes the proof of Theorem \ref{theorem4}.
\end{proof}



\bibliographystyle{apa}
\bibliography{zhang_global}

\end{document}